\documentclass[12pt]{article}
\usepackage{graphicx}
\usepackage{xcolor}
\usepackage{amssymb,amsfonts,amsthm,amsmath,calligra,tikz,tikz-cd}
\usetikzlibrary{arrows}
\usetikzlibrary{calc}
\usepackage[utf8]{inputenc}
\usepackage{accents}
\usepackage{slashed}
\usepackage{yfonts}
\usepackage{mathrsfs,pifont}
\theoremstyle{definition}
\usepackage[all]{xy}
\usepackage{mathtools}

\definecolor{myred}{RGB}{178,34,34}

\usepackage[style=numeric]{biblatex}
\usepackage[linktocpage=true]{hyperref}

\newtheorem{Definition}{Definition}
\newtheorem{Proposition}{Proposition}

\newtheorem{Theorem}{Theorem}
\newtheorem{Remark}{Remark}

\let\bs\boldsymbol

\def\llambda{{\bs \lambda}}
\def\xx{{\bs x}}

\def\uu{{\bs u}}

\def\zz{{\bs z}}

\def\sym{\text{Sym}}
\newcommand{\bL}{\reflectbox{$\mathsf{L}$}}

\newcommand{\bA}{\mathsf{A}}

\newcommand{\bT}{\mathsf{T}}

\newcommand{\lt}{\Gamma_{\llambda}}
\newcommand{\trwt}{\mathsf{W}}
\newcommand{\trwts}{\textsf{\textbf{W}}}
\newcommand{\atrwt}{\mathsf{W}^{\mathcal{A}}}

\newcommand{\tr}{\textsf{t}}
\newcommand{\trs}{\textsf{\textbf{t}}}

\newcommand{\Lie}{\mathrm{Lie}}

\newcommand{\Pic}{\mathrm{Pic}}

\newcommand{\stab}{\mathrm{Stab}}

\DeclareMathOperator\caret{\raisebox{0.75ex}{$\scriptstyle\wedge$}}

\newcommand{\pkg}[1]{\texttt{#1}}

\title{Elliptic stable envelopes of affine type A quiver varieties}
\author{Hunter Dinkins}

\begin{document}

\maketitle

\begin{abstract}
    We generalize Smirnov's formula for the elliptic stable envelopes of the Hilbert scheme of points in $\mathbb{C}^2$ to the case of affine type $A$ Nakajima quiver varieties constructed with positive stability condition. We allow for arbitrary choices of polarization and a fairly general choice of chamber. This paper is a companion to the Maple code developed by the author, which implements the formulas described in this paper and is available on the author's website\footnote{tarheels.live/dinkins/code}.
\end{abstract}

\tableofcontents

\section{Introduction}

Originally introduced by Maulik and Okounkov in \cite{MO}, stable envelopes have grown to become a rich object of study in algebraic geometry and geometric representation theory. Stable envelopes were originally defined for Nakajima quiver varieties in equivariant cohomology, though the definition has since been extended to equivariant $K$-theory and equivariant elliptic cohomology in \cite{pcmilect} and \cite{AOElliptic}. More recently, stable envelopes have been extended to the case of non-abelian group actions on quite general varieties, see \cite{indstab1} and \cite{indstab2}. 

Despite the level of generality in which elliptic stable envelopes can be defined, it is a nontrivial problem to write explicit formulas for them. 

Let us briefly survey what is known. In the original paper defining elliptic stable envelopes \cite{AOElliptic}, formulas were provided for hypertoric varieties, see also \cite{mstoric}. For cotangent bundles to partial flag varieties, formulas were written using the so-called elliptic weight functions \cite{RTV}. More generally, elliptic characteristic classes on generalized full flag varieties were studied in \cite{RWElliptic}. In the paper \cite{MirSym1}, stable envelopes for the 3d mirror dual of cotangent bundles to Grassmannians were calculated. The paper \cite{SmirnovElliptic} computes the elliptic stable envelope for the Hilbert scheme of points on $\mathbb{C}^2$, which was the first example for an affine type $A$ quiver variety. Recently, Botta provided a method to build the stable envelopes of arbitrary quiver varieties from those of quiver varieties with one framing \cite{Botta}. This, in combination with Smirnov's formula, provides formulas for the instanton moduli spaces.

The goal of this paper is to provide formulas for the elliptic stable envelopes of quiver varieties arising from finite or affine type $A$ quivers. As such, this paper generalizes all the known formulas for quiver varieties. 

Our formula closely resembles that of \cite{SmirnovElliptic} and we give a brief overview of it here. For fuller explanations of notation, see the main text. 

Let $\mathcal{M}$ be a Nakajima quiver variety arising from an affine type $A$ quiver with dimension vector $\mathsf{v}$ and framing dimension $\mathsf{w}$. There is a torus $\bT=\bA_{\mathsf{w}}\times\mathbb{C}^{\times}_{t_1} \times \mathbb{C}^{\times}_{t_2}$ which acts on $\mathcal{M}$ with finitely many fixed points. Fixed points can be indexed naturally by certain $|\mathsf{w}|$-tuples of partitions, which we write as $\llambda=(\lambda^{(i)}_{j})$. In section \ref{comb}, we describe a collection of \textit{admissible trees} inside a partition. We call the set of such trees $\Gamma_{\llambda}$. Let $\trs=(\tr^{(i)}_{j})$ be a $|\mathsf{w}|$-tuple of admissible trees such that $\tr^{(i)}_{j}$ is an admissible tree in $\lambda^{(i)}_{j}$. In section \ref{formula}, we associate a certain elliptic function 
$$
\trwt^{\lambda^{(i)}_j}_{\tr^{(i)}_{j}}:=\trwt^{\lambda^{(i)}_j}_{\tr^{(i)}_{j}}(\xx,\uu,\zz;q,t_1,t_2) 
$$
to $\tr^{(i)}_{j}$ and multiply these together to form
$$
\trwts^{\llambda}_{\trs}:=\trwts^{\llambda}_{\trs}(\xx,\uu,\zz;q,t_1,t_2):= \prod_{i=0}^{r-1} \prod_{j=1}^{\mathsf{w}_i}  \trwt^{\lambda^{(i)}_{j}}_{\tr^{(i)}_{j}}(\xx,\uu,\zz;q,t_1,t_2)
$$
The meaning of the variables are as follows:
\begin{itemize}
    \item $\xx$ stands for the collection of Chern roots of the tautological vectors bundles of $\mathcal{M}$.
    \item $\uu$ represents the equivariant parameters of the framing torus. The variables $t_1$ and $t_2$ are equivariant parameters on $\bT$ modulo the framing torus.
    \item $\zz$ stands for the K\"ahler (or dynamical) parameters, of which there is one for each vertex of the quiver.
    \item $q$ is the modular parameter of the elliptic curve defining the elliptic cohomology theory.
\end{itemize}
In addition, we define another elliptic function
$$
S_{\llambda}:=S_{\llambda}(\xx,\uu;q,t_1,t_2)
$$
which depends only on the fixed point $\llambda$ and not on the trees. All of these functions depend on  choice of chamber $\mathfrak{C}$ and polarization $T^{1/2}$. Then the main result of this paper is the following:
\begin{Theorem}[Theorem \ref{mainthm}]\label{mainthmintro}
The elliptic stable envelope of $\llambda \in \mathcal{M}^{\bT}$ is given by
$$
\stab_{T^{1/2},\mathfrak{C}}(\llambda)=\sym_{0}\,  \sym_{1} \, \ldots \, \sym_{r-1}\left( S_{\llambda} \sum_{\trs \in \lt} \trwts^{\llambda}_{\trs} \right)
$$
where $\sym_i$ is the symmetrization over the Chern roots of $\mathcal{V}_i$.
\end{Theorem}

The proof of this theorem occupies section \ref{proof} and uses an abelianization argument similar to \cite{SmirnovElliptic}. As usual, this formula can be degenerated to provide $K$-theoretic stable envelopes, see section \ref{Ktheory}.

A linear type $A$ quiver variety can be though of as an affine type $A$ quiver variety with an additional node that has dimension 0. Thus Theorem \ref{mainthmintro} also provides a formula for linear type $A$ quiver varieties. Because the torus that act on quiver varieties arising from linear type $A$ quivers differs from the affine case, a slight reparametrization must be made to write the formulas for a linear type $A$ quiver in their usual form, see section \ref{linear}.

We have also developed a Maple package \pkg{EllipticStableEnvelopes} implementing these formulas which is explained in section \ref{pkgsection}. Computational access to stable envelopes is useful pedagogically and has aided the author in uncovering new results pertaining to 3d mirror symmetry in \cite{msflag} and \cite{dinksmir4}. We hope that other researchers will similarly find this computational tool helpful.

One fascinating and presently unexplored aspect of stable envelopes pertains to Cherkis bow varieties studied in \cite{Cherk1}, \cite{Cherk2}, \cite{Cherk3}, \cite{NakBow}. Bow varieties generalize Nakajima quiver varieties in (linear and affine) type $A$ and are expected to be closed under 3d mirror symmetry. 3d mirror symmetry is a conjectural relationship between varieties arising from certain quantum field theories. One expectation of 3d mirror symmetry is that the elliptic stable envelopes of 3d mirror dual varieties should coincide in a nontrivial way, see \cite{MirSym1}, \cite{MirSym2}, and \cite{msflag}. The recent work of Rim{\'a}nyi and Shou \cite{RSbows} studied the combinatorial aspects of Cherkis bow varieties and described many of the crucial geometric ingredients needed to compute with stable envelopes. Since bow varieties are expected to be closed under 3d mirror symmetry, it is of great interest to generalize the tools used to study elliptic stable envelopes of quiver varieties to the case of bow varieties. Since much of the combinatorics of bow varieties parallels that of quiver varieties, we hope that this paper will be a step in this direction.

\subsection{Acknowledgements}
We would like to thank Andrey Smirnov, whose work on the elliptic stable envelopes of the Hilbert scheme of points in $\mathbb{C}^2$ provided the foundation for nearly all of the constructions in this paper. We also thank him for sharing Maple code, some of which was developed and incorporated into the package described here. This work was partially support by NSF grant DMS-2054527.

\subsection{Notation}

We will use the notation $\vartheta$ and $\varphi$ for the functions
\begin{align}\label{thetaphi} \nonumber
    \varphi(x)&:=\prod_{i=0}^{\infty}(1-x q^i) \\
    \vartheta(x)&:=(x^{1/2}-x^{-1/2})\varphi(qx)\varphi(q/x)
\end{align}
which satisfy the transformation properties
\begin{align}\label{thetaphi2} \nonumber
    \varphi(q x) &= (1-x)^{-1} \varphi(x) \\
    \vartheta(q x) &=-\frac{1}{\sqrt{q} x} \vartheta(x)
\end{align}

\section{Combinatorial background}\label{comb}

\subsection{Partitions}
Let $\lambda=(\lambda_1\geq \lambda_2 \geq \ldots \geq \lambda_{l}\geq 0)$ be a partition. We identify it with its Young diagram, which is the set of points
\begin{equation}\label{Yng}
\{(x,y)\in \mathbb{Z}^{2}_{> 0} \, \mid \, 1\leq x \leq l , 1 \leq y \leq \lambda_i\} 
\end{equation}
As is standard, we refer to the points in the Young diagram as ``boxes." We call the box with coordinates $(1,1)$ the corner box. Let $a\in \lambda$ be a box with coordinates $(x,y)$. The content and height of $a$ are defined to be
\begin{align*}
    c_{a}&=x-y \\
    h_{a}&=x+y-2
\end{align*}
We will often think of our Young diagrams as being rotated to the left by $45^{\circ}$. With this convention, the content and height give the horizontal and vertical coordinates, respectively, normalized so that the corner box of $\lambda$ has content 0 and height 0.

\subsection{$(\mathsf{v},\mathsf{w})$-tuples of partitions}\label{vwtuples}
Let $r \in \mathbb{N}$ and let $Q$ be the affine type $A$ quiver with vertex set $I=\{0,\ldots,r-1\}$. We always assume the arrows are oriented from left to right. Let $\mathsf{v},\mathsf{w}\in \mathbb{N}^{I}$, which will later be interpreted as the dimension and framing dimension of a Nakajima quiver variety. It will be convenient below to consider indices modulo $r$. We will denote this with an overline:
$$
\overline{i}:= i \mod r
$$

In what follows, we will consider partitions labeled by a vertex $i$ and an integer $j$ such that $1\leq j \leq \mathsf{w}_i$. For a box in such a partition, we define the height function to be the same as in the previous section. However, we translate the content function such that the content of the corner box is $i$. This means that if $(x,y)$ are the coordinates of a box $a$ in such a partition, then
$$
c_a=x-y+i
$$
We assume this convention whenever we are considering partitions indexed in this way.
\begin{Definition}\label{vwpart}
A $(\mathsf{v},\mathsf{w})$-tuple of partitions is a $(\mathsf{w}_0+\ldots \mathsf{w}_{r-1})$-tuple of partitions 
$$
\llambda=\left(\lambda^{(i)}_{j}\right)_{\substack{0 \leq i \leq r-1 \\ 1 \leq j \leq \mathsf{w}_i}}
$$
such that the total number of boxes in all partitions with content congruent to $m$ modulo $r$ is equal to $\mathsf{v}_m$ for any $m \in \{0,\ldots,r-1\}$. In other words, we require that the set
$$
\bigcup_{\substack{0 \leq i \leq r-1 \\ 1 \leq j \leq \mathsf{w}_i}}\{a \in \lambda^{(i)}_{j} \, \mid \, \overline{c_a}=m \}
$$
has size $\mathsf{v}_m$.
\end{Definition}

\begin{Remark}
As a tuple, a $(\mathsf{v},\mathsf{w})$-tuple of partitions is sensitive to the order in which the partitions appear.
\end{Remark}

It is useful to rotate the partitions and arrange them above the universal cover of the affine quiver which is the $A_{\infty}$ quiver. The vertices of the $A_{\infty}$ quiver should be identified modulo $r$ and the total number of boxes above any collection of identified vertices should equal the dimension of the vertex. See Figure \ref{fig1}.

\begin{figure}[htbp]
    \centering
\begin{tikzpicture}[scale=0.98,roundnode/.style={circle,fill,inner sep=2.5pt},refnode/.style={circle,inner sep=0pt},squarednode/.style={rectangle,fill,inner sep=3pt}] 
\node[roundnode,label=above:{6}] (N1) at (-4,0) {};
\node[roundnode,label=above:{5}] (N2) at (-3,0){};
\node[roundnode,label=above:{5}] (N3) at (-2,0){};
\node[roundnode,label=above:{4}] (N4) at (-1,0){};
\node[roundnode,label=above:{6}] (V1) at (0,0) {};
\node[roundnode,label=above:{5}] (V2) at (1,0){};
\node[roundnode,label=above:{5}] (V3) at (2,0){};
\node[squarednode,label=below:{$2$}] (F1) at (0,-1){};
\node[roundnode,label=above:{4}](V4) at (3,0){};
\node[squarednode,label=below:{$1$}] (F2) at (1,-1){};
\node[squarednode,label=below:{$1$}] (F3) at (3,-1){};
\node[roundnode,label=above:{6}] (V5) at (4,0){};
\node[roundnode,label=above:{5}] (V6) at (5,0){};
\node[roundnode,label=above:{5}] (V7) at (6,0){};
\node[roundnode,label=above:{4}] (V8) at (7,0){};

\begin{scope}[shift={(0,0)}]
\draw[thick,-](0,0)--(-4,4)--(-3,5)--(-2,4)--(-1,5)--(1,3)--(2,4)--(3,3)--(0,0);
\draw[thick,-](1,1)--(-2,4);
\draw[thick,-](2,2)--(1,3);
\draw[thick,-](-1,1)--(1,3);
\draw[thick,-](-2,2)--(0,4);
\draw[thick,-](-3,3)--(-2,4);
\end{scope}

\begin{scope}[shift={(0,4.5)}]
\draw[thick,-](0,0)--(-2,2)--(0,4)--(1,3)--(2,4)--(3,3)--(0,0);
\draw[thick,-](1,1)--(-1,3);
\draw[thick,-](2,2)--(1,3);
\draw[thick,-](-1,1)--(1,3);
\end{scope}

\begin{scope}[shift={(0,8)}]
\draw[thick,-](1,0)--(-1,2)--(0,3)--(2,1)--(1,0);
\draw[thick,-](0,1)--(1,2);
\end{scope}

\begin{scope}[shift={(0,8.5)}]
\draw[thick,-](3,0)--(1,2)--(2,3)--(3,2)--(5,4)--(6,3)--(3,0);
\draw[thick,-](4,1)--(3,2);
\draw[thick,-](5,2)--(4,3);
\draw[thick,-](2,1)--(3,2);
\end{scope}

\draw[thick, ->] (N1) -- (N2);
\draw[thick, ->] (N2) -- (N3);
\draw[thick, ->] (N3) -- (N4);
\draw[thick, ->] (N4) -- (V1);
\draw[thick, ->] (V1) -- (V2);
\draw[thick, ->] (V2) -- (V3);
\draw[thick, ->] (V3) -- (V4);
\draw[thick, ->] (V4) -- (V5);
\draw[thick, ->] (V5) -- (V6);
\draw[thick, ->] (V6) -- (V7);
\draw[thick, ->] (V7) -- (V8);
\draw[thick, ->] (V1) -- (F1);
\draw[thick, ->] (V2) -- (F2);
\draw[thick, ->] (V4) -- (F3);

% \draw[very thick,-](-1/2,-1/6)--(-1/2,1/6);
% \draw[very thick,-](7/2,-1/6)--(7/2,1/6);

\draw[dotted, -](0,0)--(0,12.5);
\draw[dotted, -](1,0)--(1,12.5);
\draw[dotted, -](2,0)--(2,12.5);
\draw[dotted, -](3,0)--(3,12.5);
\draw[dotted, -](4,0)--(4,12.5);
\draw[dotted, -](5,0)--(5,12.5);
\draw[dotted, -](6,0)--(6,12.5);
\draw[dotted, -](7,0)--(7,12.5);
\draw[dotted, -](-1,0)--(-1,12.5);
\draw[dotted, -](-2,0)--(-2,12.5);
\draw[dotted, -](-3,0)--(-3,12.5);
\draw[dotted, -](-4,0)--(-4,12.5);
\end{tikzpicture}

\begin{tikzpicture}[scale=1,roundnode/.style={circle,fill,inner sep=2.5pt},refnode/.style={circle,inner sep=0pt},squarednode/.style={rectangle,fill,inner sep=3pt}] 

 \node[roundnode,label=above:{6}] (V1) at (0,0) {};
\node[roundnode,label=above:{5}] (V2) at (1,0){};
\node[roundnode,label=above:{5}] (V3) at (2,0){};
\node[squarednode,label=below:{$2$}] (F1) at (0,-1){};
\node[roundnode,label=above:{4}](V4) at (3,0){};
\node[squarednode,label=below:{$1$}] (F2) at (1,-1){};
\node[squarednode,label=below:{$1$}] (F3) at (3,-1){};

\draw[thick, ->] (V1) -- (V2);
\draw[thick, ->] (V2) -- (V3);
\draw[thick, ->] (V3) -- (V4);
\draw[thick,dashed, ->] (V4)[bend right=60] to (V1);

\draw[thick, ->] (V1) -- (F1);
\draw[thick, ->] (V2) -- (F2);
\draw[thick,->] (V4)--(F3);
\end{tikzpicture}
 \caption{The affine type $A$ quiver with 4 vertices (bottom), along with its universal cover (top). The dimension is $\mathsf{v}=(6,5,5,4)$ and the framing dimension is $\mathsf{w}=(2,1,0,1)$.The $(\mathsf{v},\mathsf{w})$-tuple of partitions shown here is $\llambda=((4,3,1),(2,2,1),(2),(2,1,1))$.}\label{fig1}
\end{figure}
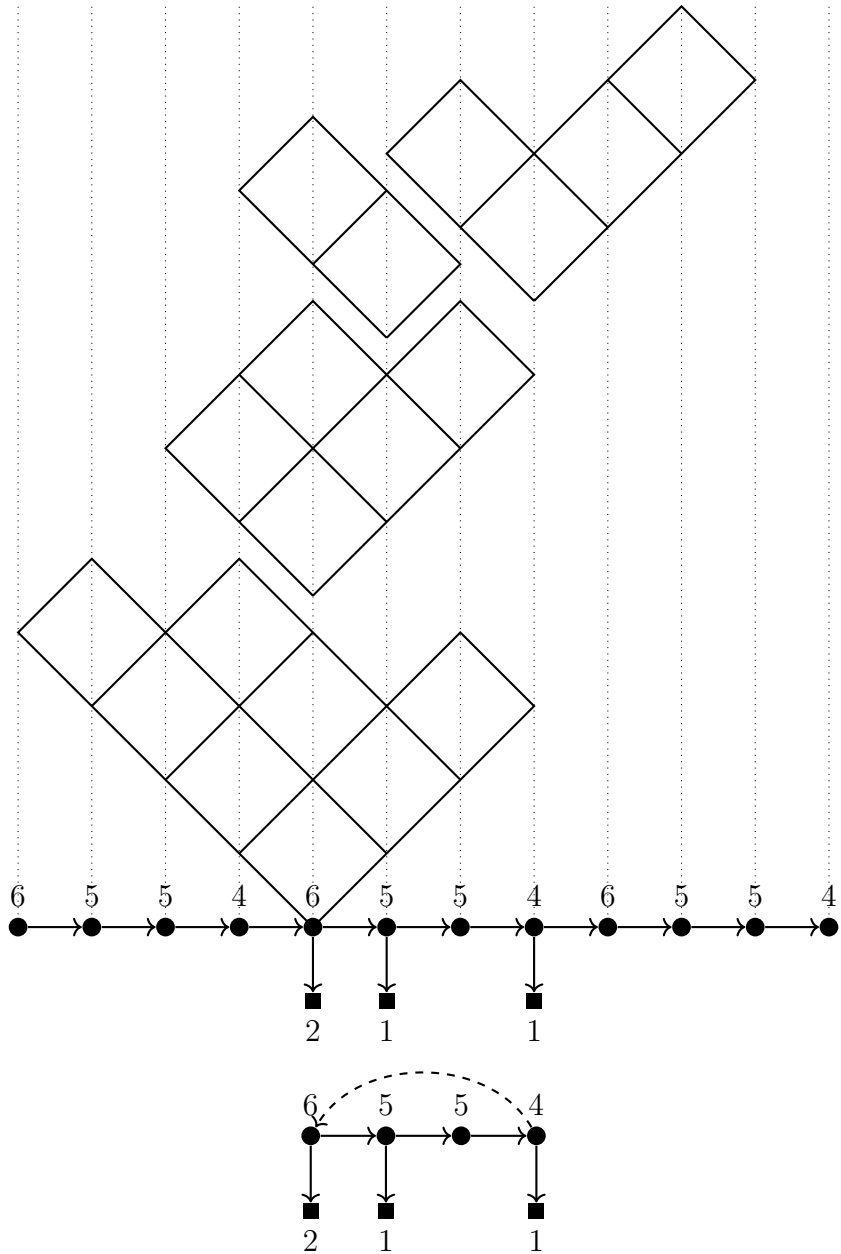

\subsection{Trees in partitions}
Two boxes in a partition $\lambda$ with coordinates $(x_1,y_1)$ and $(x_2,y_2)$ are adjacent if
$$
x_1=x_2 \text{ and } y_1=y_2 \pm 1
$$
or 
$$
y_1=y_2 \text{ and } x_1=x_2 \pm 1
$$

\begin{Definition}
A tree $\tr$ in a partition $\lambda$ is a tree (a graph with no cycles) whose vertices consist of all boxes in $\lambda$ and whose edges connect only adjacent boxes. A rooted tree in $\lambda$ is a tree in $\lambda$ with a distinguished box, called the root box.
\end{Definition}

A tree $\tr$ is totally determined by its set of edges, which, abusing notation, we will also denote by $\tr$.

An orientation on a tree is a choice of direction on each edge, or equivalently, a choice of two functions $h,t$ from edges to vertices, which sends an edge $e$ to one of the two vertices attached to it such that $h(e)\neq t(e)$. We refer to $h$ and $t$ as the head and tail functions and think of edges as proceeding from tail to head.

\begin{Definition}
A canonically oriented rooted tree in a partition $\lambda$ is a rooted tree $\tr$ in lambda with all edges oriented away from the root box. In other words, the root box does not appear in the image of the head function, and for all edges $e_1$ and $e_2$, if $h(e_1)=h(e_2)$, then $e_1=e_2$.
\end{Definition}

Any subtree of an oriented tree inherits an orientation. Given a canonically oriented rooted tree $\tr$ in $\lambda$ and a box $a \in \lambda$, there is a natural subtree of $\lambda$ rooted at $a$ with orientation induced by that of $\tr$.

\begin{Definition}
Given a rooted canonically oriented tree $\tr$ inside a partition $\lambda$ and a box $a \in \lambda$, we define $[a,\tr]$ to be the set of boxes in $\lambda$ appearing as vertices in the natural subtree of $\tr$ rooted at $a$. So $a \in [a, \tr]$, as well as any boxes that can be obtained by following edges that start at $a$ from tail to head.
\end{Definition}

\subsection{Admissible trees in partitions}
Fix a partition $\lambda$. 

\begin{Definition}
The skeleton of $\lambda$, denoted $\Sigma_{\lambda}$, is the graph with vertex set given by boxes in $\lambda$ and edge set given by all edges that connect adjacent boxes.
\end{Definition}

\begin{Definition}
A $\bL$-shaped subgraph of $\Sigma_{\lambda}$ is a subgraph of $\Sigma_{\lambda}$ with two edges, which are of the form 
$$
e_1=\{(x,y),(x+1,y)\} \quad e_2=\{(x+1,y),(x+1,y+1)\}
$$
\end{Definition}

\begin{Proposition}
There are 
$$
\sum_{i}(m_i-1)
$$
$\bL$-shaped subraphs of $\Sigma_{\lambda}$, where $m_i:=\left|\{a \in \lambda \, \mid \, c_a=i\}\right|$
\end{Proposition}
\begin{proof}
This can be proven easily by, for example, induction on the number of boxes of $\lambda$.
\end{proof}

We index the $\bL$-shaped subgraphs in $\lambda$ as $\gamma_1,\ldots, \gamma_m$. 
\begin{Proposition}
Let $e_i$, $i \in \{1, \ldots, m\}$ be edges such that $e_i$ is an edge of $\gamma_i$. Then $\Sigma_{\lambda}\setminus \{e_1,e_2,\ldots e_n\}$ is a tree in $\lambda$.
\end{Proposition}
\begin{proof}
This follows easily from the construction of $e_i$ and $\bL$-shaped subgraphs.
\end{proof}

We refer to the trees obtained as in the previous proposition as \textit{admissible trees}.
\begin{Definition}\label{adtr}
Let $\Gamma_{\lambda}$ be the set of $2^m$ canonically oriented trees in $\lambda$ rooted at the corner box whose underlying edge set forms an admissible tree.
\end{Definition}

% Alternatively, the set $\Gamma_{\lambda}$ consists of canonically oriented trees $\tr$ rooted at the corner box of $\lambda$ such that for all quadruples of boxes $(x,y)$, $(x+1,y)$, $(x,y+1)$, and $(x+1,y+1)$ in $\lambda$, $\{(x,y), (x+1,y)\}$ is an edge of $\tr$ if and only if $\{(x+1,y), (x+1,y+1)\}$ is not an edge of $\tr$.

\begin{Remark}
In what follows, we will assume that all trees in partitions are rooted at the corner box and are canonically oriented. We will henceforth refer to such trees simply as trees, and will call the corner box the root box.
\end{Remark}

Later we will need the following function:
\begin{Definition}\label{kappa}
Let $\tr$ be a tree in $\lambda$. With the identification of the Young diagram of $\lambda$ as in (\ref{Yng}), we define $\kappa(\tr)$ to be the number of vertical arrows in $\tr$ directed down plus the number of horizontal edges in $\tr$ directed left.
\end{Definition}

An example of admissible trees is given in Figure \ref{fig2}.
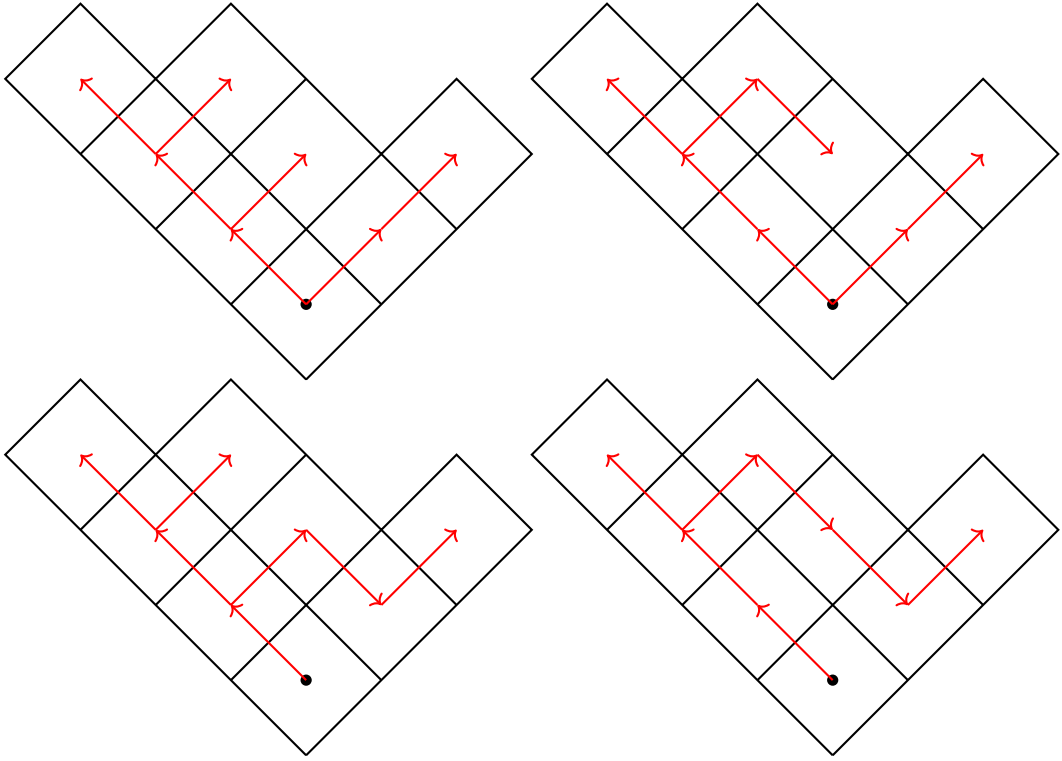
\begin{figure}[ht]
    \centering
    \begin{tikzpicture}
    
   \begin{scope}[shift={(-4,0)}]
   \node at (0,1)[circle,fill,inner sep=1.5pt]{};
\draw[thick,-](0,0)--(-4,4)--(-3,5)--(-2,4)--(-1,5)--(1,3)--(2,4)--(3,3)--(0,0);
\draw[thick,-](1,1)--(-2,4);
\draw[thick,-](2,2)--(1,3);
\draw[thick,-](-1,1)--(1,3);
\draw[thick,-](-2,2)--(0,4);
\draw[thick,-](-3,3)--(-2,4);
\draw[thick,red,->](0,1)--($(0,1)+(1,1)$);
\draw[thick,red,->]($(0,1)+(1,1)$)--($(0,1)+(2,2)$);
\draw[thick,red,->](0,1)--($(0,1)+(-1,1)$);
\draw[thick,red,->](-1,2)--($(0,1)+(0,2)$);
\draw[thick,red,->](-1,2)--(-2,3);
\draw[thick,red,->](-2,3)--(-3,4);
\draw[thick,red,->](-2,3)--(-1,4);
\end{scope}

   \begin{scope}[shift={(3,0)}]
    \node at (0,1)[circle,fill,inner sep=1.5pt]{};
\draw[thick,-](0,0)--(-4,4)--(-3,5)--(-2,4)--(-1,5)--(1,3)--(2,4)--(3,3)--(0,0);
\draw[thick,-](1,1)--(-2,4);
\draw[thick,-](2,2)--(1,3);
\draw[thick,-](-1,1)--(1,3);
\draw[thick,-](-2,2)--(0,4);
\draw[thick,-](-3,3)--(-2,4);
\draw[thick,red,->](0,1)--(1,2);
\draw[thick,red,->](1,2)--(2,3);
\draw[thick,red,->](0,1)--(-1,2);
\draw[thick,red,->](-1,4)--(0,3);
\draw[thick,red,->](-1,2)--(-2,3);
\draw[thick,red,->](-2,3)--(-3,4);
\draw[thick,red,->](-2,3)--(-1,4);
\end{scope}

 \begin{scope}[shift={(-4,-5)}]
  \node at (0,1)[circle,fill,inner sep=1.5pt]{};
\draw[thick,-](0,0)--(-4,4)--(-3,5)--(-2,4)--(-1,5)--(1,3)--(2,4)--(3,3)--(0,0);
\draw[thick,-](1,1)--(-2,4);
\draw[thick,-](2,2)--(1,3);
\draw[thick,-](-1,1)--(1,3);
\draw[thick,-](-2,2)--(0,4);
\draw[thick,-](-3,3)--(-2,4);
\draw[thick,red,->](0,3)--(1,2);
\draw[thick,red,->](1,2)--(2,3);
\draw[thick,red,->](0,1)--(-1,2);
\draw[thick,red,->](-1,2)--(0,3);
\draw[thick,red,->](-1,2)--(-2,3);
\draw[thick,red,->](-2,3)--(-3,4);
\draw[thick,red,->](-2,3)--(-1,4);
\end{scope}

 \begin{scope}[shift={(3,-5)}]
  \node at (0,1)[circle,fill,inner sep=1.5pt]{};
\draw[thick,-](0,0)--(-4,4)--(-3,5)--(-2,4)--(-1,5)--(1,3)--(2,4)--(3,3)--(0,0);
\draw[thick,-](1,1)--(-2,4);
\draw[thick,-](2,2)--(1,3);
\draw[thick,-](-1,1)--(1,3);
\draw[thick,-](-2,2)--(0,4);
\draw[thick,-](-3,3)--(-2,4);
\draw[thick,red,->](0,3)--(1,2);
\draw[thick,red,->](1,2)--(2,3);
\draw[thick,red,->](0,1)--(-1,2);
\draw[thick,red,->](-1,4)--(0,3);
\draw[thick,red,->](-1,2)--(-2,3);
\draw[thick,red,->](-2,3)--(-3,4);
\draw[thick,red,->](-2,3)--(-1,4);
\end{scope}

\end{tikzpicture}
    \caption{The four admissible trees inside the partition $(4,3,1)$. The root box is indicated by the black dot, and the canonical orientation is shown. From top to bottom and left to right, the values of the function $\kappa$ are $1, -1, -1$, and $1$.}
    \label{fig2}
\end{figure}

\section{Affine type \texorpdfstring{$A$}{A} quiver varieties}
\subsection{Definition of quiver varieties}\label{defs}
Fix $r \in \mathbb{N}$. In this paper, we study the quiver varieties associated to the affine type $A$ quiver $Q$ with vertex set $I=\{0,1,\ldots,r-1\}$ and edges $i\to i+1$. In this section, we understand all indices to be taken modulo $r$.

Let $\mathsf{v},\mathsf{w}\in \mathbb{Z}^{I}_{\geq 0}$ and fix complex vector spaces $V_i$ and $W_i$ such that $\dim V_i=\mathsf{v}_i$ and $\dim W_i=\mathsf{w}_i$. Let
$$
M:=M(\mathsf{v},\mathsf{w})= \bigoplus_{i=0}^{r-1} Hom(V_i, V_{i+1})  \oplus \bigoplus_{i=0}^{r-1} Hom(W_i,V_i)
$$ 
The group $G_{\mathsf{v}}=\prod_{i=0}^{r-1} GL(V_i)$ acts naturally on $M$ and induces an action on $T^*M$. We identify
\begin{multline*}
    T^*M= M \oplus M^* = \bigoplus_{i=0}^{r-1} Hom(V_i, V_{i+1})  \oplus \bigoplus_{i=0}^{r-1} Hom(V_{i+1}, V_{i})  \\ \oplus \bigoplus_{i=0}^{r-1} Hom(W_i,V_i)  \bigoplus_{i=0}^{r-1} Hom(V_i,W_i)
\end{multline*}
using the trace pairing.

With respect to this decomposition, we denote elements of $T^*M$ by $(X_i,Y_i,I_i,J_i)_{i \in I}$, where
\begin{align*}
    X_i&:V_i \to V_{i+1} \\
    Y_i&: V_{i+1} \to V_{i} \\
    I_i&: W_i \to V_i \\
    J_i&: V_i \to W_i
\end{align*}
Then $(g_i) \in G_{\mathsf{v}}$ acts by 
$$
(g_i) \cdot (X_i,Y_i,I_i,J_i) = (g_{i+1}X_i g_{i}^{-1},g_{i}Y_{i} g_{i+1}^{-1}, g_i I_i, J_i g_i^{-1})
$$
The action of $G_{\mathsf{v}}$ induces a moment map
$$
\mu: T^*M \to \mathfrak{g}_{\mathsf{v}}^*, \quad \text{where}\quad  \mathfrak{g}_{\mathsf{v}}=\Lie(G_{\mathsf{v}})
$$
We choose the cocharacter of $G_{\mathsf{v}}$ given by
$$
\theta: G_v \to \mathbb{C}^{\times}, \quad (g_i) \mapsto \prod_{i \in I} \det g_i
$$
\begin{Definition}[\cite{Nak1}, \cite{NakALE}, \cite{GinzburgLectures}]
The quiver variety associated to the choice of $r$, $\mathsf{v}$, $\mathsf{w}$, and $\theta$ is the geometric invariant theory quotient
$$
\mathcal{M}_{\theta}(\mathsf{v},\mathsf{w}):= \mu^{-1}(0)/\!\!/\!\!_{\theta} G_{\mathsf{v}} = \mu^{-1}(0)^{\theta-ss}/G_{\mathsf{v}}
$$
\end{Definition}
In this paper, we will only work with the positive stability condition $\theta$. When the choice of dimensions $\mathsf{v}$ and $\mathsf{w}$ are fixed, we will often denote the quiver variety by $\mathcal{M}(\mathsf{v},\mathsf{w})$ or just $\mathcal{M}$ with the understanding that the stability condition $\theta$ is fixed as in this section.

\begin{Proposition}[\cite{GinzburgLectures} Proposition 5.1.5] \label{stability}
A tuple $(X_i,Y_i,I_i,J_i)\in T^*M$ is $\theta$-semistable if and only if there is no collection of proper subspaces $(S_i)_{i \in I}$ of $(V_i)_{i \in I}$ stable under $X_i$ and $Y_i$ containing the image of $I_i$.
\end{Proposition}

\subsection{Torus action}
The variety $\mathcal{M}_{\theta}(\mathsf{v},\mathsf{w})$ comes equipped with several torus actions. We define the framing torus to be the maximal torus
$$
\bA_{\mathsf{w}} \subset \prod_{i \in I} GL(W_i)
$$
The framing torus naturally acts on $M(\mathsf{v},\mathsf{w})$, which descends to an action on $\mathcal{M}_{\theta}(\mathsf{v},\mathsf{w})$ that preserves the symplectic form. We denote the coordinates on the framing torus as $u^{(i)}_{j}$ where $i \in I$ and $1\leq j \leq \mathsf{w}_i$. 

In addition, the torus $\mathbb{C}^{\times}_{t_1}\times \mathbb{C}^{\times}_{t_2}$ with coordinates $(t_1,t_2)$ acts on $T^*M$ by 
$$
(t_1,t_2) \cdot (X_i,Y_i,I_i,J_i) = (t_2 X_i, t_1 Y_i,I_i,t_1 t_2 J_i)
$$
% Passing to a double cover of $\mathbb{C}_{t_1}^{\times} \times \mathbb{C}_{t_2}^{\times}$, we introduce coordinates $a$ and $\hbar$ such that 
% $$
% t_1 = \sqrt{\hbar} a, \quad t_2 = \sqrt{\hbar}/a
% $$
% and define
% $$
% \bT:=\bA_{\mathsf{w}}\times \mathbb{C}^{\times}_a \times \mathbb{C}^{\times}_{\hbar}
% $$
% The torus $\bT$ acts on $\mathcal{M}_{\theta}(\mathsf{v},\mathsf{w})$ such that the $\bT$-weight of the symplectic form is $\hbar^{-1}$.
The $\mathbb{C}^{\times}_{t_1}\times\mathbb{C}^{\times}_{t_2}$-weight of the symplectic form on $\mathcal{M}$ is $\hbar^{-1}$ where $\hbar=t_1t_2$. The torus $\text{ker}(\hbar^{-1}) \subset \mathbb{C}^{\times}_{t_1}\times\mathbb{C}^{\times}_{t_2}$ is 1-dimensional. Passing to a double cover, we let $a$ be the coordinate on $\text{ker}(\hbar^{-1})$ such that
$$
t_1 = \sqrt{\hbar} a, \quad t_2 = \sqrt{\hbar}/a
$$
We define
$$
\bT:=\bA_{\mathsf{w}}\times \mathbb{C}^{\times}_{t_1} \times \mathbb{C}^{\times}_{t_2}
$$

With this notation, the subtorus of $\bT$ preserving the symplectic form is $\bA_{\mathsf{w}}\times \mathbb{C}^{\times}_a$.

\subsection{Tautological bundles}
The vector spaces $V_i$ descend to topologically nontrivial tautological bundles on $\mathcal{M}$, defined by
\begin{equation}\label{tb}
\mathcal{V}_i:=\mu^{-1}(0)^{\theta-ss} \times_{G_{\mathsf{v}}} V_i
\end{equation}
The $\bT$ action on  $M$ induces a $\bT$-equivariant structure on the bundles $\mathcal{V}_i$. It was proven in \cite{kirv} that the tautological bundles generate the $\bT$-equivariant $K$-theory of $\mathcal{M}$. Using the vector spaces $W_i$, we also obtain a collection of topologically trivial bundles $\mathcal{W}_i$.

\subsection{Torus fixed points}\label{fixedpoints}
Let $p \in \mathcal{M}^{\bT}$. Then there is a representative for $p$ of the form $(X_i,Y_i,I_i,J_i)$ such that for all $t=(u_{\mathsf{w}},t_1,t_2) \in \bA_{\mathsf{w}}\times \mathbb{C}^{\times}_{t_1} \times \mathbb{C}^{\times}_{t_2}$, there exists some $g_{t} \in G_{\mathsf{v}}$ where
$$
t \cdot (X_i,Y_i,I_i,J_i) = g_{t} \cdot (X_i,Y_i,I_i,J_i)
$$
Suppressing indices on $g_{t}$, this means that
$$
(t_2 X_i, t_1 Y_i, I_i u_{\mathsf{w}}^{-1}, t_1t_2 u_{\mathsf{w}} J_i)= (g_{t} X_i g_{t}^{-1}, g_{t} Y_i g_{t}^{-1},g_{t} I_i, J_i g_{t}^{-1})
$$
Decomposing each $V_i$ into generalized eigenspaces of $g_{t}$ as $V_i=\bigoplus_{\zeta} V_i(\zeta)$, we see that
\begin{align*}
    I_i: \text{Span}_{\mathbb{C}}\{ e^{(i)}_{j}\}\to V_i((u^{(i)}_{j})^{-1})
\end{align*}
where $e^{(i)}_{j}$ is the coordinate vector of $W_i$ with $\bA_{\mathsf{w}}$-weight of $(u^{(i)}_{j})^{-1}$. In addition, the maps $X_i$ and $Y_i$ change the weights by $t_2$ and $t_1$, respectively:
\begin{align*}
    X_i:& V_i(\zeta) \to V_{i+1}(t_2 \zeta) \\
    Y_i:& V_{i+1}(\zeta) \to V_{i}(t_1 \zeta)
\end{align*}
By Proposition \ref{stability}, the images of $I_i$ generate all of $\bigoplus_{i \in I} V_i$ under the action of $X_i$ and $Y_i$. This implies that each $V_i$ decomposes into eigenspaces for $g_{t}$ with eigenvectors of the form $t_1^{a} t_2^{b} (u^{(i)}_{j})^{-1}$ with $a,b,\geq 0$. In addition
$$
t_1 t_2 u_{\mathsf{w}}J_i = J_i g_t^{-1} \implies J_i v = 0 \text{ unless } v \in V_i((t_1 t_2 u^{(i)}_{j})^{-1})
$$
Such eigenvectors never appear, which means that $J_i=0$. Since $(X_i,Y_i,I_i,J_i)\in \mu^{-1}(0)$, the explicit form of the moment map implies that 
$$
X_{i} Y_{i}-Y_{i+1}X_{i+1}=0
$$

This data provides us with a partition $\lambda^{(i)}_{j}=(\lambda^{(i)}_{j,1}\geq \lambda^{(i)}_{j,2},\ldots)$ for each framing weight $u^{(i)}_{j}$, where $\lambda^{(i)}_{j,k}$ is the number of eigenvectors with eigenvalue $t_1^{a} t_2^{k-1} u^{(i)}_{j}$ for some $a \in \mathbb{Z}_{\geq 0}$. It follows immediately from semistability and the moment map condition that $\lambda^{(i)}_{j}$ is indeed a partition. Under this interpretation, the boxes of the Young diagram are in a natural bijection with all the vectors with $\bA_{\mathsf{w}}$-weight of $(u^{(i)}_{j})^{-1}$. Furthermore, the collection of Young diagrams $\lambda^{(i)}_{j}$ respects the dimension $\mathsf{v}$, by which we mean that the number of boxes in all the Young diagrams representing vectors in $V_i$ is equal to $\mathsf{v}_i$. In other words, we obtain a $(\mathsf{v},\mathsf{w})$-tuple of partitions in the sense of Definition \ref{vwpart}.

Conversely, given a $(\mathsf{v},\mathsf{w})$-tuple of partitions $\llambda$, we can construct an element  of $T^*M$ by letting $V_i$ be the vector space with basis given by the boxes of $\llambda$ with content modulo $r$ equal to $i$. We define $X_i$ as the map that sends each box to the box directly to the right of it, or 0 if no such box exists. Similarly, $Y_i$ sends each box to the box above it if such a box exists, and is 0 otherwise. We define $I_i$ as the map that sends each coordinate vector of $\mathbb{C}^{\mathsf{w}_i}$ to the different root boxes of the partitions $\lambda^{(i)}_{j}$ for $j=1,\ldots,\mathsf{w}_i$. Finally, we define $J_i=0$. Then it is easy to check that the tuple $(X_i,Y_i,I_i,J_i)$ defines a point of $\mu^{-1}(0)^{\theta-ss}$.

To summarize, we have
\begin{Proposition}
$\bT$-fixed points on $\mathcal{M}(\mathsf{v},\mathsf{w})$ are naturally indexed by $(\mathsf{v},\mathsf{w})$-tuples of partitions.
\end{Proposition}

\subsection{Weights at torus fixed points}\label{tweights}
Let $\llambda$ be a $(\mathsf{v},\mathsf{w})$-tuple of partitions, which we think of as indexing a $\bT$-fixed point on $\mathcal{M}(\mathsf{v},\mathsf{w})$. The argument in the previous subsection also provides the $\bT$-weights of the tautological bundle $\mathcal{V}_i$ at the fixed point $\llambda$. If we denote the coordinates of boxes in $\llambda$ by $(x,y)$, then 
$$
\mathcal{V}_i = \sum_{i \in I} \sum_{j=1}^{\mathsf{w}_i} \sum_{\substack{\square=(x,y) \in \lambda^{(i)}_{j} \\ \overline{c_{\square}}=i}} t_1^{1-y} t_2^{1-x} u^{(i)}_{j} \in K_{\bT}(pt)
$$

\subsection{Polarization}
A polarization of $\mathcal{M}$ is a class $T^{1/2}\in  K_{\bT}(\mathcal{M})$ such that the class of the tangent bundle of $\mathcal{M}$ can be decomposed as
$$
T \mathcal{M} = T^{1/2} + t_1 t_2 (T^{1/2})^{\vee} \in K_{\bT}(\mathcal{M})
$$
The tangent bundle can be expressed in terms of the tautological bundles as
\begin{align}\label{tan}
T\mathcal{M}=& \sum_{i=0}^{r-1} \left(t_2 Hom(V_i,V_{i+1}) + t_1 Hom(V_{i+1},V_i) \right) \\ &+ \sum_{i=0}^{r-1} \left(Hom(\mathcal{W}_i,\mathcal{V}_i) + t_1 t_2 Hom(\mathcal{V}_i,\mathcal{W}_i)\right) \\ &- \sum_{i=0}^{r-1} \left( Hom(V_i,V_i)+t_1t_2 Hom(V_i,V_i)\right)
\end{align}

There are many natural polarizations of $\mathcal{M}$ which can be obtained by choosing one of the two terms from each summation above. In this paper, we allow for an arbitrary choice of polarization.

\subsection{Chern roots and torus weights labeled by boxes}
Fix $\llambda \in \mathcal{M}^{\bT}$. For each $i$, choose some bijection
$$
\{\text{Chern roots of } \mathcal{V}_i\} \longleftrightarrow \{a \in \llambda \, \mid \, \overline{c_a}=i\}
$$
Since our eventual formula for the elliptic stable envelope will involve a symmetrization over all the Chern roots of each tautological bundle, the choice of bijection does not matter. Hence we have a Chern root $x_a$ for each box $a \in \llambda$. Additionally, there is a natural bijection
$$
\{a \in \llambda \, \mid \, \overline{c_a}=i\} \longleftrightarrow \{\bT-\text{weights of } \mathcal{V}_i|_{\llambda}\}
$$
which sends the box $a=(x,y)\in \lambda^{(i)}_{j}$ to the weight $\varphi^{\llambda}_{a}:=t_1^{1-y} t_2^{1-x} u^{(i)}_{j}$.

% The choice of bijection will not affect our formula. With respect to these bijections, for each $a\in \llambda$, we let $\varphi^{\llambda}_{a}$ be the restriction of the Chern root $x_{a}$ to the $\bT$-fixed point $\llambda$.

% As in section \ref{tweights}, if $a \in \llambda$ is a box with coordinates $(x,y)$ in a partition $\lambda^{(i)}_{j}$ corresponding to the $j$th framing at the $i$th node, then
% $$
% \varphi^{\llambda}_{a}=t_1^{1-y} t_2^{1-x} u^{(i)}_{j}
% $$

\subsection{Chamber}
We allow for a fairly general choice of chamber $\mathfrak{C} \subset \Lie_{\mathbb{R}}(\bA)$, where $\bA$ is the subtorus of $\bT$ preserving the symplectic form. A chamber is represented by a real cocharacter 
$$
\sigma \in \text{cochar}(\bA)\otimes_{\mathbb{Z}} \mathbb{R}
$$
Recall that $\bA=\bA_{\mathsf{w}}\times \mathbb{C}^{\times}_{a}$, where $\bA_{\mathsf{w}}$ is the framing torus, and $a=\sqrt{t_1/t_2}$. With respect to this decomposition,
$$
\sigma =(\sigma_{\mathsf{w}},\sigma_a)
$$
We assume that our cocharacter is generic, in the sense that it lies off of the walls determined by the $\bT$-weights of the tangent spaces at the fixed points:
$$
\sigma \in \Lie_{\mathbb{R}}(\bA)\setminus\{\sigma' \, \mid \, \langle \sigma', w \rangle =0, w \in \text{char}_{\bT}(T_{\llambda}\mathcal{M}), \llambda \in \mathcal{M}^{\bT}\}
$$
where $\langle -,-\rangle$ is the natural pairing on characters and cocharacters.

Such a cocharacter decomposes the tangent space at each fixed point into attracting and repelling directions:
$$
T_{\llambda}\mathcal{M} = N_{\llambda}^{+} \oplus N_{\llambda}^{-}, \quad N_{\llambda}^{\pm}=\sum_{\substack{w \in \text{char}_{\bT}(T_{\llambda}\mathcal{M}) \\ \pm \langle \sigma, w \rangle >0}} w
$$
In other words, $N_{\llambda}^{\pm}$ consists of weights $w$ of the tangent space so that $\lim_{t^{\pm}\to 0} t^{\langle \sigma, w \rangle}$ exists.

We assume that $\sigma_{a}$ is infinitesimally small compared to the rest of the components, by which we mean that the presence of $a$ in a weight never affects its attracting/repelling properties unless it is the only coordinate that appears. In other words, we assume that the $\sigma$-attracting/repelling properties of tangent weights with nonzero $\bA_{\mathsf{w}}$-weight are totally determined by $\sigma_{\mathsf{w}}$.

\begin{Definition}\label{admisschamb}
A chamber $\mathfrak{C}$ satisfying the above property is called an admissible chamber.
\end{Definition}

We can uniformly rescale each component of our cocharacter $\sigma$ without changing the corresponding chamber. So, without loss of generality, we assume that $\sigma_a=1$.

\subsection{Ordering on boxes}
Given a box $a \in \llambda$, we define
\begin{equation}\label{rho}
\rho_{a}:= \langle \sigma, \varphi^{\llambda}_a \rangle + \epsilon \deg_{\hbar} (\varphi^{\llambda}_a)
\end{equation}
where $0<\epsilon\ll 1$. Due to the genericity assumption on $\sigma$, the function $\rho$ defines a total order on the set of all boxes in $\llambda$.

The function $\rho$ is simply a convenient way to extend the partial order of attraction on weights induced by $\sigma$ to a total order on boxes. In particular, for $\frac{\varphi^{\llambda}_a}{\varphi^{\llambda}_{b}} \in \text{char}_{\bT}(T_{\llambda}\mathcal{M})$
$$
\frac{\varphi^{\llambda}_a}{\varphi^{\llambda}_{b}} \text{ is } \sigma\text{-repelling} \implies \rho_a < \rho_b
$$

\subsection{Orientation between vertices}\label{orient}

Recall that $t_1= \sqrt{\hbar} a$ and $t_2=\sqrt{\hbar}/a$. The cocharacter $\sigma$ forces one of these $\bT$-weights to be attracting and one to be repelling. We denote by $t_{+}$ the attracting one and $t_{-}$ the repelling one. 

This induces on orientation between adjacent vertices in the quiver, given by the following ``tail" and ``head" functions:
\begin{align*}
    t(i,i+1)&=\begin{cases}
  i &  t_{-}=t_2 \\
   i+1 & t_{-}=t_1
    \end{cases} \\
h(i,i+1)&=  \begin{cases}i+1& t_{-}=t_2 \\
  i & t_{-}=t_1
    \end{cases}
\end{align*}
% \begin{align*}
%     t(\overline{i},\overline{i+1})&=\begin{cases}
%   \overline{i} &  t_{-}=t_2 \\
%   \overline{i+1} & t_{-}=t_1
%     \end{cases} \\
% h(\overline{i},\overline{i+1})&=  \begin{cases}\overline{i+1}& t_{-}=t_2 \\
%   \overline{i} & t_{-}=t_1
%     \end{cases}
% \end{align*}
where the indices are taken modulo $r$.

Since, roughly speaking, $t_2$ (resp. $t_1$) acts on arrows pointed to the right (resp. left) in the framed doubled quiver, this can be though of as distinguishing the arrows acted on by the repelling direction.

\section{Formula for elliptic stable envelopes}\label{formula}

We will not define elliptic stable envelopes here. Rather, we refer to \cite{AOElliptic}, \cite{indstab1}, \cite{indstab2}, and \cite{SmirnovElliptic}. Elliptic stable envelopes of a quiver variety $\mathcal{M}$ depend on a choice of chamber $\mathfrak{C} \subset \text{cochar}(A)\otimes_{\mathbb{Z}}\mathbb{R}$ and polarization $T^{1/2}$. 

Since all quiver varieties considered in this paper have finitely many torus fixed points, we follow the convention described in section 2.13 of \cite{SmirnovElliptic} and think of the elliptic stable envelope of a fixed point as providing a section $\stab_{\mathfrak{C},T^{1/2}}(\llambda)$ of a line bundle over $\mathsf{E}_{\bT}(\mathcal{M})$, the extended equivariant elliptic cohomology scheme of $\mathcal{M}$. This scheme has the form
$$
\mathsf{E}_{\bT}(\mathcal{M})=\left(\bigsqcup_{\llambda \in \mathcal{M}^{\bT}} \widehat{\mathsf{O}}_{\llambda}\right) / \Delta
$$
where each $\widehat{\mathsf{O}}_{\llambda}$ is isomorphic to the base of the extended elliptic cohomology theory and $\Delta$ denotes the gluing data. The base of the theory is the product of $\text{rank}(\bT)+r$ copies of the elliptic curve $E=\mathbb{C}^{\times}/q^{\mathbb{Z}}$, where $r$ is the number of vertices of the quiver. We can use the coordinate on $\mathbb{C}^{\times}$ to provide a ``coordinate" on $E$. The first $\text{rank}(\bT)$ coordinates naturally correspond to the equivariant parameters, whereas we denote the latter by $z_0,z_1,\ldots,z_{r-1}$ and refer to them as K\"ahler parameters.

Restricting the section $\stab_{\mathfrak{C},T^{1/2}}(\llambda)$ to one of the components $\widehat{\mathsf{O}}_{\llambda'}$ provides a section of a line bundle over a product of elliptic curves, which can thus be described in terms of theta functions of the equivariant and K\"ahler parameters.

Our formula below will be given in the off-shell form, which means that $\stab_{\mathfrak{C},T^{1/2}}(\llambda)$ will be described in terms of the Chern roots of the tautological bundles of $\mathcal{M}$. Restricting the Chern roots to their values at another fixed point $\llambda'$ provides $\stab_{\mathfrak{C},T^{1/2}}(\llambda)|_{\llambda'}$.

For the remainder of this section, we fix $r \in \mathbb{N}$, $I=\{0,1,\ldots,r-1\}$, and $\mathsf{v},\mathsf{w} \in \mathbb{N}^{I}$. For the corresponding quiver variety $\mathcal{M}$, we choose a polarization $T^{1/2}$, and an admissible chamber $\mathfrak{C}$ as in Definition \ref{admisschamb}. We fix $\llambda \in \mathcal{M}^{\bT}$. As in (\ref{rho}), we have a real-valued function $\rho$ on boxes in $\llambda$. The chamber also provides $t_{+}$, $t_{-}$, and the functions $h$ and $t$ from section \ref{orient}.

We remind the reader that all trees $\tr$ inside partitions are assumed to be admissible in the sense of Definition \ref{adtr} and canonically oriented away from the root box.

\subsection{Statement of the formula}\label{formula2}
The $(\mathsf{v},\mathsf{w})$-tuple of partitions $\llambda$ consists of partitions $\lambda^{(i)}_{j}$ for $i \in I$ and $1 \leq j \leq \mathsf{w}_i$. Each of these partitions has a root box, which we denote by $r_{i,j}$.

We define the following three functions:
\begin{align*}
S_1(\xx;q,t_1,t_2)&=\prod_{i=0}^{r-1} \prod_{\substack{a,b, \in \llambda \\ \overline{c_a}=t(\overline{i},\overline{i+1}) \\ \overline{c_b}=h(\overline{i},\overline{i+1}) \\ \rho_a+1 < \rho_b}} \vartheta( t_{+} x_a/x_b) \prod_{\substack{a,b, \in \llambda \\ \overline{c_a}=t(\overline{i},\overline{i+1}) \\ \overline{c_b}=h(\overline{i},\overline{i+1}) \\ \rho_b < \rho_a +1}} \vartheta( t_{-}x_b/x_a) \\
S_2(\xx,\uu;q,t_1,t_2)&= \prod_{i=0}^{r-1} \prod_{j=1}^{\mathsf{w}_i}\left( \prod_{\substack{a \in \llambda \\ \overline{c_a}= i  \\ \rho_a \leq \rho_{r_{i,j}} }} \vartheta(x_a/u^{(i)}_{j}) \prod_{\substack{a \in \llambda \\ \overline{c_a}=i \\ \rho_{r_{i,j}} < \rho_a }} \vartheta(t_1 t_2 u^{(i)}_{j}/x_a)\right)\\
S_3(\xx;q,t_1,t_2)&= \prod_{\substack{a,b\in\llambda \\ \overline{c_a}=\overline{c_b} \\ \rho_a<\rho_b}} \vartheta(x_a/x_b) \vartheta(t_1 t_2 x_a/x_b)
\end{align*}
Let
$$
S_{\llambda}(\xx,\uu;q,t_1,t_2)=\frac{S_1(\xx;q,t_1,t_2) S_2(\xx,\uu;q,t_1,t_2)}{S_3(\xx;q, t_1,t_2)}
$$
We will sometimes omit the arguments and just write $S_{\llambda}$.

\begin{Definition}
Let $\llambda \in \mathcal{M}^{\bT}$. Let
$$
T^{1/2}=T^{1/2}_{<0} + T^{1/2}_{=0} + T^{1/2}_{>0}
$$
be the decomposition of $T^{1/2}$ into virtual sub-bundles that restrict to repelling, stationary, and attracting directions at $\llambda$. We define the index of $\llambda$ to be
$$
\text{ind}_{\llambda}= T^{1/2}_{>0}
$$
\end{Definition}

The index can be expressed in terms of the Chern roots of the tautological bundles, and we define integers $d^{\llambda}_a$ by
$$
\det \text{ind}_{\llambda}= \prod_{a \in \llambda} x_a^{d^{\llambda}_a}
$$

\begin{Definition}
For a tree $\tr^{(i)}_j$ in a partition $\lambda^{(i)}_j$, we define
\begin{multline*}
\trwt^{\lambda^{(i)}_j}_{\tr^{(i)}_{j}}(\xx,\uu,\zz;q,t_1,t_2) = (-1)^{\kappa(\tr^{(i)}_{j})}\phi\left(\frac{x_{r_{i,j}}}{\varphi^{\llambda}_{r_{i,j}}}, \prod_{a \in [r_{i,j}, \tr^{(i)}_j]} z_{c_a} (t_1t_2)^{d^{\llambda}_a}\right) \\ 
\prod_{e \in \tr^{(i)}_j} \phi\left(\frac{x_{h(e)} \varphi^{\llambda}_{t(e)}}{x_{t(e)}\varphi^{\llambda}_{h(e)}}, \prod_{a \in [h(e),\tr^{(i)}_{j}]} z_{c_a} (t_1t_2)^{d^{\llambda}_a}\right)
\end{multline*}
where $\kappa$ is the function in Definition \ref{kappa}. Here, the product $\prod\limits_{e \in \tr^{(i)}_{j}}$ denotes the product over edges in the tree $\tr^{(i)}_{j}$.
\end{Definition}

Let
$$
\lt:=\prod_{i=0}^{r-1} \prod_{j=1}^{\mathsf{v}_i} \Gamma_{\lambda^{(i)}_{j}}
$$
where $\Gamma_{\lambda^{(i)}_{j}}$ is from Definition \ref{adtr} so that an element of $\lt$ is a tuple of trees, one inside of each partition in $\llambda$.

\begin{Definition}
For $\llambda \in X^{\bT}$ and $\trs \in \lt$, let 
$$
\trwts^{\llambda}_{\trs}(\xx,\uu,\zz;q,t_1,t_2)= \prod_{i=0}^{r-1} \prod_{j=1}^{\mathsf{w}_i}  \trwt^{\lambda^{(i)}_{j}}_{\tr^{(i)}_{j}}(\xx,\uu,\zz;q,t_1,t_2)
$$
\end{Definition}
We will sometimes omit the arguments and just write $\trwts^{\llambda}_{\trs}$.

Our main theorem is:
\begin{Theorem}\label{mainthm}
The elliptic stable envelope of $\llambda \in \mathcal{M}^{\bT}$ is given by
$$
\stab_{T^{1/2},\mathfrak{C}}(\llambda)=\sym_{0}\,  \sym_{1} \, \ldots \, \sym_{r-1}\left( S_{\llambda} \sum_{\trs \in \lt} \trwts^{\llambda}_{\trs} \right)
$$
where $\sym_i$ is the symmetrization over the Chern roots of $\mathcal{V}_i$.
\end{Theorem}

\subsection{Linear quivers}\label{linear}
Let $I=\{0,1,\ldots,r-1\}$ and let $\mathsf{v},\mathsf{w}\in \mathbb{Z}^{I}_{\geq 0}$. The \textit{linear} type $A_r$ quiver varieties constructed from this data is canonically isomorphic to the affine type $\hat{A}_{r+1}$ quiver variety constructed by taking the dimension and framing dimension at the last vertex to be 0, while the rest of the dimensions are the same as $\mathsf{v}$ and $\mathsf{w}$.

As a result, Theorem \ref{mainthm} also applies for linear type $A$ quiver varieties. However, the torus acting on the $\hat{A}_{r+1}$ variety has one more dimension than that acting on the $A_{r}$ variety. This extra action can be compensated for by the gauge group, which resolves this discrepancy. 

This means that for Theorem \ref{mainthm} to match the existing formulas in the literature for linear type $A$ quiver varieties, we must consider the subtorus $\bT'\subset \bT$ defined by setting $t_2=1$. Then the linear maps corresponding to right pointing arrows in the quiver are unchanged by the torus, while linear maps corresponding to left pointing arrows are scaled by $t_1$. In this case, the $\bT'$-weight of the symplectic form is $t_1^{-1}$, which is typically denoted by $\hbar^{-1}$. Hence, to obtain formulas which agree with the existing formulas in the literature for linear quiver varieties, we must simply substitute $t_2=1, t_1=\hbar$ in Theorem \ref{mainthm}.

\subsection{$K$-theory limit}\label{Ktheory}

It is simple to obtain formulas for stable envelopes in $K$-theory from Theorem \ref{mainthm}. We explain this procedure here. 

In addition to their dependence on a chamber and polarization, $K$-theoretic stable envelopes depend on a choice of slope $s$, which is a generic element of 
$$
\Pic(\mathcal{M})\otimes_{\mathbb{Z}}\mathbb{R}
$$
Here, generic means that $s$ lies in the complement of a certain collection of affine hyperplanes, see \cite{OS}. 

Let $\stab_{T^{1/2},\mathfrak{C}}^{s}(\llambda)$ be the $K$-theoretic stable envelope of $\llambda$ with respect to the slope $s$. Choose a total order on the $\bT$-fixed points of $\mathcal{M}$: $\llambda_1, \llambda_2, \ldots \llambda_n$. With respect to this ordered basis, we denote the restriction matrices of the elliptic and $K$-theoretic stable envelopes by
\begin{align*}
    \mathsf{E}_{i,j}=\stab_{T^{1/2},\mathfrak{C}}(\llambda_i)\big|_{\llambda_j} \\
     \mathsf{K}^{s}_{i,j}=\stab^{s}_{T^{1/2},\mathfrak{C}}(\llambda_i)\big|_{\llambda_j}
\end{align*}
We also introduce the matrix
$$
\mathsf{D}=\text{diag}\left(\left(\det T^{1/2}|_{\llambda_1}\right)^{-1/2}, \ldots, \left(\det T^{1/2}|_{\llambda_n}\right)^{-1/2}\right)
$$

It is known that the tautological line bundles $\mathcal{L}_i:=\det \mathcal{V}_i$ generate $\Pic(\mathcal{M})$, see \cite{kirv}. So we can express $s$ (possibly non-uniquely) in terms of this generating set as $(s_0,\ldots,s_{r-1})$ where $s_i \in \mathbb{R}$. We denote $z q^{s}$ for $(z_0 q^{s_0},\ldots, z_{r-1} q^{s_{r-1}})$. Then we have:

\begin{Theorem}[\cite{AOElliptic} Proposition 4.3]
The $K$-theoretic stable envelopes can be recovered from the elliptic stable envelopes as follows:
$$
\mathsf{K}^{s}=\lim_{q\to 0} \left( \mathsf{D} \cdot \mathsf{E}|_{z=zq^{-s}} \right)
$$
\end{Theorem}

More generally, \cite{KS2} describes the limit on the right hand side of the previous theorem for non-generic slopes $s$. In this case, some of the K\"ahler parameters survive. Nevertheless, the equivariant and K\"ahler parameters separate in a nontrivial way. The resulting limit factors to a product of matrices given in terms of the $K$-theoretic stable envelope of $\mathcal{M}$ for a generic slope $s'$ near $s$ and the $K$-theoretic stable envelope of a variety $X^{!}_s$ for small ample slope. The variety $X^{!}_s$ is a subvariety of the 3d mirror dual variety of $\mathcal{M}$ determined by the non-generic slope $s$.

\section{Proof of Theorem \ref{mainthm}}\label{proof}
Fix $\mathsf{v},\mathsf{w}$, and $\mathcal{M}$ as before. The proof of Theorem \ref{mainthm} uses abelianization, which is the concept of relating quotients by reductive groups to quotients by a maximal torus, see \cite{ProudHyper}. Abelianization has already appeared in the context of stable envelopes in \cite{AOElliptic}, \cite{Shenfeld}, and \cite{SmirnovElliptic}.

\subsection{Abelianization}
In this section, we use the notation of section \ref{defs}. The character $\theta: G_{\mathsf{v}}\to \mathbb{C}^{\times}$ induces a character of the diagonal maximal torus $S_{\mathsf{v}}\subset G_{\mathsf{v}}$, which we also denote by $\theta$.

\begin{Definition}
The abelianization of $\mathcal{M}_{\theta}(\mathsf{v},\mathsf{w})$ is the algebraic symplectic reduction
$$
\mathcal{AM}_{\theta}(\mathsf{v},\mathsf{w}):= T^*M(\mathsf{v},\mathsf{w})/\!\!/\!\!/\!\!/\!\!/_{\theta} S_{\mathsf{v}}
$$
where $S_{\mathsf{v}}\subset G_{\mathsf{v}}$ is the diagonal maximal torus.
\end{Definition}
When the choice of $\mathsf{v}$ and $\mathsf{w}$ are clear, we abbreviate the abelianzation by $\mathcal{AM}$.

\subsection{Abelianization of fixed points}\label{abel}
Let $\llambda$ be a $(\mathsf{v},\mathsf{w})$-tuple of partitions indexing a $\bA$-fixed point of $\mathcal{M}$, and let $\lambda\in \llambda$ be a partition. Define
$$
V_{\lambda}^{i}=\mathbb{C}^{d^{i}_{\lambda}}
$$
where $d^{i}_{\lambda}$ is the number of boxes of $\lambda$ with content $i$, where we understand the content to be shifted relative to the framing to which $\lambda$ corresponds, as in section \ref{vwtuples}.

Consider
$$
M_{\lambda}=\bigoplus_{i=0}^{r-1} Hom(V_{\lambda}^i,V_{\lambda}^{i+1})\oplus Hom(\mathbb{C},V^{r_{\lambda}}_{\lambda})
$$
where $r_{\lambda}$ denotes the content of the root box of $\lambda$ and we assume that indices are taken mod $r$

The group $G_{\lambda}:=\prod_{i}GL(V_{\lambda}^{i})$ acts naturally on $M_{\lambda}$, and we define
$$
\mathcal{M}_{\lambda}:=T^* M_{\lambda} /\!\!/\!\!/\!\!/_{\theta_{\lambda}} G_{\lambda}
$$
where $\theta_{\lambda}:( g_i)_{i \in I} \mapsto \prod_{i \in I} \det g_i$. The varieties $\mathcal{M}_{\lambda}$, where $\lambda$ ranges over all partitions, account for all nonempty finite type $A$ quiver varieties with 1 framing dimension. Such varieties are geometrically just a single point, and the enumerative geometry of quasimaps to these varieties has been extensively studied in \cite{dinksmir2}, \cite{dinksmir3}, and \cite{dinkms1}.

\begin{Definition}
The abelianization of $\lambda$ is the hypertoric variety
$$
\mathcal{A}_{\lambda}:=T^*M_{\lambda}/\!\!/\!\!/\!\!/_{\theta_{\lambda}} S_{\lambda}
$$
where $S_{\lambda} \subset G_{\lambda}$ is the diagonal maximal torus.
\end{Definition}

Since $\llambda$ is a $(\mathsf{v},\mathsf{w})$-tuple of partitions, we have
\begin{equation}\label{boxdec}
\bigoplus_{\lambda \in \llambda} V_{\lambda}^{i} = \mathbb{C}^{\mathsf{v}_i}
\end{equation}
Identifying each $V_{\lambda}^{i}$ as a subspace of $\mathbb{C}^{\mathsf{v}_i}$, we have a natural action of the group $S_{\mathsf{v}}$ on 
$$
\mathsf{M}_{\llambda}:=\bigoplus_{\lambda \in \llambda} M_{\lambda}
$$

% With respect to this action, the data described at the end of section \ref{fixedpoints} gives the unique semistable representative of the only point in
% $$
% T^* \mathsf{M}_{\llambda} /\!\!/\!\!/\!\!/_{\theta} G_{\mathsf{v}}= \{pt\}
% $$
% But we can also consider the symplectic reduction with respect to the maximal torus $S_{\mathsf{v}} \subset G_{\mathsf{v}}$.

\begin{Definition}
The abelianization of $\llambda$ is the algebraic symplectic reduction
$$
\mathcal{AM}_{\llambda}:=T^*\mathsf{M}_{\llambda} /\!\!/\!\!/\!\!/_{\theta} S_{\mathsf{v}}
$$
\end{Definition}

Due to the natural inclusion $T^*\mathsf{M}_{\llambda} \subset T^*M(\mathsf{v},\mathsf{w})$ at the level of vector spaces, the hypertoric variety $\mathcal{AM}_{\llambda}$ is embedded as an $\bA$-fixed component of $\mathcal{AM}$.

Since the action of $S_{\mathsf{v}}$ on $\mathsf{M}_{\llambda}$ preserves each direct summand, it follows immediately from the definitions that there is a natural identification
\begin{equation}\label{abfac}
\mathcal{AM}_{\llambda}= \prod_{\lambda \in \llambda} \mathcal{A}_{\lambda}
\end{equation}

\subsection{Martin diagram}
In this subsection, we denote by $\mu_{G}: T^*M \to \mathfrak{g}_{\mathsf{v}}^{*}$ the moment map for the $G_{\mathsf{v}}$ action on $M$ and by $\mu_{S}:T^*M \to \mathfrak{s}^*$ the moment map for the $S_{\mathsf{v}}$ action on $M$. Similarly, we write the chosen $G_{\mathsf{v}}$ character as $\theta_{G}$ and the induced $S_{\mathsf{v}}$ character as $\theta_{S}$, whereas we previously denoted both by $\theta$. 

The natural inclusion $S_{\mathsf{v}} \hookrightarrow G_{\mathsf{v}}$ induces a map 
$$
\iota^*: \mathfrak{g}_{\mathsf{v}}^* \to \mathfrak{s}^*
$$
and it is known that
$$
\mu_{S}=\iota^* \circ \mu_G
$$

Let $B \subset G_{\mathsf{v}}$ be a Borel subgroup such that $S_{\mathsf{v}} \subset B$. Fix a maximal compact subgroup $U$ of $G_{\mathsf{v}}$ and a $U$ invariant Hermitian metric on $M$. This provides a moment map
$$
\mu_{U}=(\mu_{U,\mathbb{R}},\mu_{U,\mathbb{C}}): T^{*}M \to \mathfrak{u}^{*} \oplus \mathfrak{u}_{\mathbb{C}}^{*}, \quad \mathfrak{u}=\Lie(U)
$$
The hyperk\"ahler description of the quiver variety is given by
$$
\mathcal{M}\cong \mu_{U,\mathbb{R}}^{-1}(\eta) \cap \mu_{U,\mathbb{C}}^{-1}(0)/U
$$
where $\eta$ is the differential of the pullback of $\theta_{G}$ under $U\hookrightarrow G_{\mathsf{v}}$. There is a natural map
$$
\pi:\mu_{U,\mathbb{R}}^{-1}(\eta) \cap \mu_{U,\mathbb{C}}^{-1}(0)/(U\cap S) \to \mathcal{M}
$$
obtained by taking the quotient by the larger group. The fibers of this map are flag varieties.

Let $\mathfrak{b}$ be the Lie algebra of $B$. Since $\mathfrak{s} \subset \mathfrak{b}$, if $\mu_{G}(p) \in \mathfrak{b}^{\perp}$, then $\mu_{S}(p)=0$. This provides a map
$$
\mu^{-1}_{G}(\mathfrak{b}^{\perp}) \to \mu^{-1}_{S}(0)
$$
Since this map is $S_{\mathsf{v}}$ equivariant and preserves $\theta_S$ stability, it descends to the quotient
$$
\mathsf{j}_-:\mu^{-1}_{G}(\mathfrak{b}^{\perp})^{\theta_S-ss}/ S_{\mathsf{v}} \to \mathcal{AM}
$$ 

Recall that $\mu_{U,\mathbb{C}}$ is actually just $\mu_{G}$, the usual moment map for $G_{\mathsf{v}}$. And since $0 \in \mathfrak{g}^*$ vanishes on $\mathfrak{b}$, we get a map
$$
\mu^{-1}_{U,\mathbb{C}}(0) \to \mu^{-1}_{G}(\mathfrak{b}^{\perp})
$$
Also, since $p\in \mu_{U,\mathbb{R}}^{-1}(\eta)$ implies that $p$ is $\theta_{G}$-stable, and since $\theta_{G}$-stability implies $\theta_{S}$-stability, we get a map
$$
\mathsf{j}_{+}:\mu^{-1}_{U,\mathbb{R}}(\eta) \cap \mu^{-1}_{U,\mathbb{C}}(0)/(U \cap S_{\mathsf{v}}) \to  \mu^{-1}_{G}(\mathfrak{b}^{\perp})^{\theta_{S}-ss}/S_{\mathsf{v}}
$$

Overall, this provides us with a diagram
\begin{equation*}
\begin{tikzcd}
\mathsf{Fl}:=\mu_{U,\mathbb{R}}^{-1}(\eta)\cap \mu_{U,\mathbb{C}}^{-1}(0)/(U \cap S_{\mathsf{v}}) \arrow[r,"\mathsf{j}_{+}"] \arrow[d,"\pi"] & \mu^{-1}_{G}(\mathfrak{b}^{\perp})^{\theta_S-ss}/ S_{\mathsf{v}} \arrow[r,"\mathsf{j}_{-}"] & \mathcal{AM} \\ 
\mathcal{M}
\end{tikzcd}
\end{equation*}
as in section 4.3.1 of \cite{AOElliptic}.

For a fixed point $\llambda \in \mathcal{M}^{\bA}$, let $\mathsf{Fl}'$ be the component of $\pi^{-1}(\llambda)^{\bA}$ such that the normal $\bA$-weights to $\mathsf{Fl}'$ in $\pi^{-1}(\llambda)$ are repelling. Then we have a similar diagram:
\begin{equation*}
\begin{tikzcd}
\mathsf{Fl}' \arrow[r,"\mathsf{j}_{+}'"] \arrow[d,"\pi'"] & T^*\mathsf{M}_{\llambda} \cap \mu^{-1}_{G}(\mathfrak{b}^{\perp})^{\theta_S-ss}/ S_{\mathsf{v}} \arrow[r,"\mathsf{j}_{-}'"] & \mathcal{AM}_{\llambda} \\ 
\{\llambda\}
\end{tikzcd}
\end{equation*}

\subsection{Construction of stable envelopes}\label{sheafdiagram}

This gives us the following map of sheaves:
\begin{equation*}
\begin{tikzcd}[column sep=huge]
\mathscr{U}'\arrow{r}{ \mathsf{j}'_{- *} \circ ({\mathsf{j}'_{+}}^{*})^{-1}\circ {\pi'_{*}}^{-1}} \arrow{d}{\text{Stab}_{\mathfrak{C},T^{1/2}}}& \Theta\left(T^{1/2}\mathcal{AM}_{\llambda}\right)\otimes \mathscr{U}'\arrow{d}{\text{Stab}'_{\mathfrak{C},T^{1/2}_{\mathcal{AM}}}} \\
\Theta\left(T^{1/2}\mathcal{M}\right)\otimes \mathscr{U} & \Theta\left(T^{1/2}\mathcal{AM}\right)\otimes \mathscr{U}\arrow[swap]{l}{\pi_{*}\circ \mathsf{j}_{+}^{*} \circ \mathsf{j}_{- *}^{-1}}
\end{tikzcd}
\end{equation*}
where $\mathscr{U}$ is the universal line bundle on the extended elliptic cohomology of $\mathcal{M}$, $\mathscr{U}'$ is the shift of the universal line bundle  on the extended elliptic cohomology of $\mathcal{M}^{\bA}$, $\Theta$ stands for the elliptic Thom class, and $T^{1/2}_{\mathcal{AM}}$ is a suitable chosen polarization of $\mathcal{AM}$, see sections 2.7.3, 3.3.2, and 2.6.3 in \cite{AOElliptic}.

In the above diagram, we implicitly understand all maps to be maps of the pushfowards of the corresponding sheaves to the base of the extended elliptic cohomology. A-priori, the base of the extended elliptic cohomology of $\mathcal{AM}$ differs from that of $\mathcal{M}$. But as discussed in section 7.3 of \cite{SmirnovElliptic}, the natural map from characters of $G_{\mathsf{v}}$ to characters of $S_{\mathsf{v}}$ allows one to resolve this discrepancy. Concretely, the K\"ahler parameters of $\mathcal{AM}$ correspond bijectively to the $1$-dimensional coordinate subspaces of $V_i$ for $i \in I$. For a fixed $\llambda \in \mathcal{M}^{\bA}$, we identify the K\"ahler parameters of $\mathcal{AM}$ with boxes in $\llambda$ using the decomposition (\ref{boxdec}). These are mapped to the K\"ahler parameters of $\mathcal{M}$ by
$$
z_{a} \mapsto z_{c_a} \text{ where } a \in \lambda \text{ and }  \lambda \in  \llambda
$$

In the original paper \cite{AOElliptic}, elliptic stable envelopes for Nakajima varieties are defined with reference to the above diagram as
\begin{equation}\label{stabdef}
\stab_{\mathfrak{C},T^{1/2}}(\llambda)=\pi_{*} \circ \mathsf{j}_{+}^{*} \circ (\mathsf{j}_{-*})^{-1} \circ \stab'_{\mathfrak{C},T^{1/2}_{\mathcal{AM}}} \circ \mathsf{j}_{-*}' \circ ({\mathsf{j}'_{+}}^{*})^{-1} \circ {\pi'_{*}}^{-1}
\end{equation}

There are some technical points with defining the inverses in the above formula which are discussed in sections 4.3.11 and 4.3.12 of \cite{AOElliptic}.

\subsection{}
Recall from section \ref{abel} that for each $\lambda \in \llambda$ we have varieties $\mathcal{M}_{\lambda}$ with corresponding abelianizations $\mathcal{A}_{\lambda}$. For each of these, there exists a Martin diagram, and the corresponding maps $\pi'_{\lambda}$, $\mathsf{j}'_{\lambda,+}$, and $\mathsf{j}'_{\lambda,-}$, along with their pushforwards and pullbacks.

\begin{Proposition}\label{factor}
Under the natural identifications induced by $\mathcal{AM}_{\llambda} = \prod_{\lambda \in \llambda} \mathcal{A}_{\lambda}$, each of the maps $\pi'_{*}$, ${\mathsf{j}'_{+}}^{*}$, and $\mathsf{j}'_{-*}$ factors into the direct product of maps:
$$
\pi'_{*}=\prod_{\lambda \in \llambda} \pi'_{\lambda,*}, \quad {\mathsf{j}'_{+}}^{*}= \prod_{\lambda \in \llambda} {\mathsf{j}'_{\lambda,+}}^{*}, \quad \mathsf{j}'_{-*}= \prod_{\lambda \in \llambda} \mathsf{j}'_{\lambda,-*}
$$
\end{Proposition}
\begin{proof}
This follows from the decomposition at the level of prequotient data
$$
M_{\llambda}= \bigoplus_{\lambda \in \llambda} M_{\lambda}
$$
from section \ref{abel}.
\end{proof}

\subsection{Distinguished collection of fixed points}
Let $\lambda$ be a partition. Points in $\mathcal{A}_{\lambda}$ are represented by tuples of linear maps. We identify basis vectors of $\bigoplus_{i} V^{i}_{\lambda}$ with boxes in $\lambda$ so that an element of $M_{\lambda}$ is given by a collection of matrices $(X_i)_{i=0,\ldots,r-1}$, where $X_i \in Hom(\mathbb{C}^{d^{i}_{\lambda}},\mathbb{C}^{d^{i+1}_{\lambda}})$, along with $I \in Hom(\mathbb{C},V^{r_{\lambda}}_{\lambda})$. 

Given a tree $\tr$ in $\lambda$, we define an action of a torus $\mathbb{C}^{\times}_{\tr}$ on $M_{\lambda}$ as follows. Let $\epsilon$ be the coordinate on $\mathbb{C}^{\times}_{\tr}$. If $a, b \in \lambda$ are two boxes such that $c_b=c_a+1$ and the oriented edge $a\to b$ is in the tree $\tr$, then the $(b,a)$-entry of $X_{c_a}$ is scaled by $\epsilon^{h_b-h_a}$. All other entries in the matrices, as well as the map $I$, are left unchanged.

We now describe a fixed point of the $\mathbb{C}^{\times}_{\tr}$ action on $T^*M_{\lambda}$. With respect to the chosen identification of basis vectors with boxes in $\lambda$, the torus $S_{\lambda}$ has a coordinate for each box of $\lambda$. We consider the compensating map given in coordinates by 
$$
\epsilon \to (\epsilon^{h_a})_{a \in \lambda}
$$

For $a,b \in \lambda$ with $c_b=c_a+1$, we write the $(b,a)$th entry of $X_{c_a}$ as $(X_{c_a})_{a,b}$. Consider the decomposition
$$
M_{\lambda}=M_0\oplus M_1
$$
where $M_0$ consists of matrices $X_i$ such that
$$
(X_i)_{a,b}\neq 0 \implies a \to b \in \tr
$$
along with the maps $I$ which send $1$ to a multiple of the root box. The space $M_1$ consists of matrices $X_i$ such that
$$
(X_i)_{a,b}\neq 0 \implies a \to b \notin \tr
$$
along with the maps $I$ have $0$ in the root box component. By the general theory of hypertoric varieties, a representative of a fixed point defined by this data is given by an element of $T^*M_0$ with exactly $|\lambda|$ nonzero components.

In summary, we have described a torus action and a fixed point to a tree $\tr$ in $\lambda$. We will abuse notation and refer to this point simply as $\tr$.

\subsection{Stable envelopes of trees}
Given $\llambda$ and a tuple of trees $\trs \in \lt$ such that $\tr_{\lambda}$ denotes the tree in the partition $\lambda$, we let 
$$
\bT_{\trs}:= \prod_{\lambda \in \llambda} \mathbb{C}^{\times}_{\tr_{\lambda}}
$$
We understand that $\mathbb{C}^{\times}_{\tr_{\lambda}}$ acts naturally on the factor $\mathcal{A}_{\lambda}$ inside of $\mathcal{AM}_{\llambda}=\prod_{\lambda \in \llambda} \mathcal{A}_{\lambda}$. Abusing notation, we denote by $\trs$ the $\bT_{\trs}$-fixed point $\prod_{\lambda} \tr_{\lambda} \subset \mathcal{AM}_{\llambda}$.

Let $\mathfrak{C}''_{\tr_{\lambda}}\subset \Lie_{\mathbb{R}}(\mathbb{C}^{\times}_{\tr_{\lambda}})$ denote the chamber of $\mathbb{C}^{\times}_{\tr_{\lambda}}$ made up of real cocharacters that pair positively with the coordinate $\epsilon$ on $\mathbb{C}^{\times}_{\tr_{\lambda}}$ described in the previous subsection. 

Since we have 
$$
\Lie_{\mathbb{R}}(\bT_{\trs})=\bigoplus_{\lambda \in \llambda} \Lie_{\mathbb{R}}(\mathbb{C}^{\times}_{\tr_{\lambda}})
$$ 
we let $\mathfrak{C}''$ be the chamber of $\bT_{\trs}$ whose elements are made up of positive real linear combinations of the elements of $\mathfrak{C}''_{\lambda}$ for each $\lambda \in \llambda$. 

% real cocharacters that pair positively with the coordinate $\epsilon$ on each $\mathbb{C}^{\times}_{\tr_{\lambda}}$ described in the previous section.

Let $\mathfrak{C}'$ be the chamber of $\bA \times \bT_{\trs}$ made up of positive real linear combinations of elements of $\mathfrak{C}$ and $\mathfrak{C}''$ (recall that $\mathfrak{C}$ is the chosen chamber of the torus $\bA$).

The fixed point $\trs$ can be thought of as a $\bT_{\trs}$-fixed point of $\mathcal{AM}_{\llambda}$, or as a $\bA \times \bT_{\trs}$-fixed point of $\mathcal{AM}$. We will the stable envelope of $\trs$ in both cases.

% Fix a polarization of $\mathcal{M}$. From this, we can obtain polarizations of the various abelianizations as discussed in \cite{AOElliptic} section 4.3.3. In the stable envelopes of the abelianizations discussed in the Proposition below, we assume

Recall from section \ref{formula2} the functions
$$
S_1(\xx;q,t_1,t_2)=\prod_{i=0}^{r-1} \prod_{\substack{a,b, \in \llambda \\ \overline{c_a}=t(\overline{i},\overline{i+1}) \\ \overline{c_b}=h(\overline{i},\overline{i+1}) \\ \rho_a+1 < \rho_b}} \vartheta( t_{+} x_a/x_b) \prod_{\substack{a,b, \in \llambda \\ \overline{c_a}=t(\overline{i},\overline{i+1}) \\ \overline{c_b}=h(\overline{i},\overline{i+1}) \\ \rho_b < \rho_a +1}} \vartheta( t_{-}x_b/x_a) 
$$
and
$$
S_2(\xx,\uu;q,t_1,t_2)= \prod_{i=0}^{r-1} \prod_{j=1}^{\mathsf{w}_i}\left( \prod_{\substack{a \in \llambda \\ \overline{c_a}= i  \\ \rho_a \leq \rho_{r_{i,j}} }} \vartheta(x_a/u^{(i)}_{j}) \prod_{\substack{a \in \llambda \\ \overline{c_a}=i \\ \rho_{r_{i,j}} < \rho_a }} \vartheta(t_1 t_2 u^{(i)}_{j}/x_a)\right)
$$

For each $\lambda \in \llambda$, let $S_1^{\lambda}(\xx;q,t_1,t_2)$ (resp. $S_2^{\lambda}(\xx,\uu;q,t_1,t_2)$) consist of the terms of $S_1(\xx;q,t_1,t_2)$ (resp. $S_2(\xx,\uu;q,t_1,t_2)$) that only involve products over boxes in $\lambda$. So
$$
S_1^{\lambda}(\xx;q,t_1,t_2)=\prod_{i=0}^{r-1} \prod_{\substack{a,b, \in \lambda \\ \overline{c_a}=t(i,i+1) \\ \overline{c_b}=h(i,i+1) \\ \rho_a+1 < \rho_b}} \vartheta( t_{+} x_a/x_b) \prod_{\substack{a,b, \in \lambda \\ \overline{c_a}=t(i,i+1) \\ \overline{c_b}=h(i,i+1) \\ \rho_b < \rho_a +1}} \vartheta( t_{-}x_b/x_a) 
$$
and
$$
S_2^{\lambda}(\xx,\uu;q,t_1,t_2)=\prod_{i=0}^{r-1} \prod_{j=1}^{\mathsf{w}_i}\left( \prod_{\substack{a \in \lambda \\ \overline{c_a}= i  \\ \rho_a \leq \rho_{r_{i,j}} }} \vartheta(x_a/u_{i,j}) \prod_{\substack{a \in \lambda \\ \overline{c_a}=i \\ \rho_{r_{i,j}} < \rho_a }} \vartheta(t_1 t_2 u_{i,j}/x_a)\right)
$$
which we abbreviate by $S_1^{\lambda}$ and $S_2^{\lambda}$.

\begin{Proposition}\label{trstab}
Up to a polarization dependent shift of the K\"ahler parameters by a power of $t_1 t_2$, the elliptic stable envelope of the $\bA \times \bT_{\trs}$-fixed point $\trs \in \mathcal{AM}$ is
$$
\text{Stab}_{\mathfrak{C}'}(\trs)=S_{1} S_{2} \prod_{\lambda \in \llambda} \atrwt_{\tr_{\lambda}}
$$
where
$$
\atrwt_{\tr_{\lambda}} = (-1)^{\kappa(\tr_{\lambda})} \phi\left(\frac{x_{r_{\lambda}}}{\varphi^{\llambda}_{r_{\lambda}}}, \prod_{a \in [r_{\lambda},\tr_{\lambda}]} z_{a} (t_1 t_2)^{d^{\lambda}_{a}}\right) \prod_{e \in \tr} \phi\left(\frac{x_{h(e)} \varphi^{\llambda}_{t(e)}}{x_{t(e)} \varphi^{\llambda}_{h(e)}}, \prod_{a\in [h(e),\tr_{\lambda}]} z_a (t_1 t_2)^{d^{\llambda}_{a}} \right)
$$
and $r_{\lambda}$ denotes the root box of $\lambda$.

Up to a polarization dependent shift of the K\"ahler parameters by a power of $\hbar$, the elliptic stable envelope of the $\bT_{\trs}$-fixed point $\trs \in \mathcal{AM}_{\llambda}$ is
$$
\text{Stab}_{\mathfrak{C}''}(\trs)= \prod_{\lambda \in \llambda} \text{Stab}_{\mathfrak{C}_{\lambda}''}(\tr_{\lambda})
$$
where 
$$
\text{Stab}_{\mathfrak{C}''_{\lambda}}(\tr_{\lambda}) =  S^{\lambda}_1 S^{\lambda}_2 \atrwt_{\tr_{\lambda}}
$$
\end{Proposition}
\begin{Remark}
The careful reader will notice that we did not specify a polarization for the varieties $\mathcal{AM}$ or $\mathcal{AM_{\llambda}}$. This is justified by the fact that stable envelopes for different polarizations differ from each other only by a shift of the K\"ahler parameters, see section 3.3.7 in \cite{AOElliptic}. 
\end{Remark}
\begin{proof}
By the canonical decomposition (\ref{abfac}), the stable envelope of the fixed point $\trs$ is the product of the stable envelopes of $\tr_{\lambda}$ for each $\lambda \in \llambda$. The latter were calculated in Proposition 6 of \cite{SmirnovElliptic} using the known formulas for stable envelopes of hypertoric varieties. Our result differs slightly from Smirnov's as we allow for more general chambers. A straightforward modification of the proof of Proposition 6 in \cite{SmirnovElliptic} gives our result.
\end{proof}

\begin{Proposition}\label{formalinverse}
$$
\pi'_{*} \circ {\mathsf{j}'_{+}}^{*} \circ (\mathsf{j}'_{-*})^{-1}\left(\sum_{\trs \in \lt} \text{Stab}_{\mathfrak{C}''}(\trs) \right)=1
$$
\end{Proposition}
\begin{proof}
From Proposition \ref{factor}, we have
$$
\pi'_{*} \circ {\mathsf{j}'_{+}}^{*} \circ (\mathsf{j}'_{-*})^{-1}\left(\sum_{\trs \in \lt} \text{Stab}_{\mathfrak{C}''}(\trs) \right) = \prod_{\lambda \in \llambda} \pi'_{\lambda,*} \circ {\mathsf{j}'_{\lambda,+}}^{*} \circ (\mathsf{j}'_{\lambda,- *})^{-1}\left(\sum_{\tr \in \Gamma_{\lambda}} \stab_{\mathfrak{C}''_{\lambda}}(\tr_{\lambda})\right)
$$
By Theorem 5 in \cite{SmirnovElliptic}, for each $\lambda$, 
$$
\sum_{\tr \in \Gamma_{\lambda}} \stab_{\mathfrak{C}''_{\lambda}}(\tr_{\lambda})=1
$$
which gives the result.
\end{proof}

\subsection{Conclusion of proof}
Recall from section \ref{sheafdiagram} that K\"ahler parameters of $\mathcal{AM}$ must be identified with those of $\mathcal{M}$. In particular, this is equivalent to the substitution $z_a \mapsto z_{c_a}$ for $a \in \lambda$, $\lambda \in \llambda$.

From Proposition \ref{formalinverse} and (\ref{stabdef}), we have
$$
\stab_{\mathfrak{C},T^{1/2}}(\llambda)=\pi_{*} \circ \mathsf{j}^{*}_{+} \circ (\mathsf{j}_{-*})^{-1} \circ \stab'_{\mathfrak{C}}\left(\sum_{\trs \in \lt} \stab_{\mathfrak{C}''}(\trs) \right)
$$
The triangle lemma for stable envelopes, see \cite{AOElliptic} section 3.1, says that the composition of two stable envelopes is another stable envelope. In particular, we have 
$$
\stab_{\mathfrak{C},T^{1/2}}(\llambda)=\pi_{*} \circ \mathsf{j}^{*}_{+} \circ (\mathsf{j}_{-*})^{-1} \circ \left( \sum_{\trs \in \lt} \stab_{\mathfrak{C}'}(\trs) \right)
$$

For the rest of the maps, we have
\begin{Proposition}\label{maps}

Let $f(\xx)$ be a section of $\Theta(T^{1/2} \mathcal{AM})\otimes \mathscr{U}$. Then
$$
\pi_{*} \circ \mathsf{j}^*_{+} \circ (\mathsf{j}_{-*})^{-1}\left(f(\xx)\right)=\sym_{0}\,  \sym_{1} \ldots \sym_{r-1} \frac{f(\xx)}{\prod\limits_{\substack{a,b \in \llambda \\ \overline{c_a}=\overline{c_b }\\ \rho_a<\rho_b}} \vartheta\left(x_a/x_b\right)\vartheta\left(t_1 t_2 x_a/x_b\right)}
$$
% For the other map, we have 
% $$
% \pi'_{*} \circ \mathsf{j}'^*_{+} \circ (\mathsf{j}'_{-*}
% )^{-1}\left(f(\xx)\right)=\sym_{\llambda} \frac{f(\varphi^{\llambda})}{\prod\limits_{\lambda \in \llambda} \prod\limits_{\substack{a,b \in \llambda \\ c_a = c_b \\ \rho_a<\rho_b }} \vartheta\left(\varphi^{\llambda}_a/\varphi^{\llambda}_b\right)\vartheta\left(t_1 t_2 \varphi^{\llambda}_a/\varphi^{\llambda}_b\right) }
% $$
% (the denominator is the $\bA$-fixed part of the previous denominator, then restricted to the fixed point) and $\sym_{\llambda}$ permutes the Chern roots, while respecting both the bundle as well as the partition that they came from.
\end{Proposition}
\begin{proof}
The proof of this result is analogous to that of \cite{SmirnovElliptic} Proposition 7.
\end{proof}

By Propositions \ref{trstab} and \ref{maps}, we deduce Theorem \ref{mainthm} up to a polarization dependent shift of the K\"ahler parameters by $t_1 t_2$. Repeating the argument of section 8.3 in \cite{SmirnovElliptic}, one can see that we have set up our formulas so that no further shifts are necessary.

\section{Maple Package \pkg{EllipticStableEnvelope}}\label{pkgsection}
In this final section, we explain the implementation of the formulas in this paper. The Maple package \pkg{EllipticStableEnvelope} is available at the author's website\footnote{\url{www.tarheels.live/dinkins/code}} or on the Maple cloud. In what follows, Maple input will be written using the red {\color{myred} \pkg{typewriter}} font, and Maple output will be written in {\color{blue} blue}. 

In what follows, we will use the following names for the different pieces of data needed:
\begin{align*}
    &\pkg{v}: \text{dimension vector of the quiver variety} \\
    &\pkg{w}: \text{ framing dimension vector} \\
    &\pkg{arrows}: \text{ data determining the polarization} \\
    &\pkg{chamb}: \text{ chamber}
\end{align*}

The procedures provided by the package are \pkg{Attracting}, \pkg{ChamberExample}, \pkg{FixedPoints}, \pkg{NormalForm}, \pkg{PolExample}, \pkg{Polarization}, \pkg{Quasiperiods}, \pkg{Repelling}, \pkg{Restrict}, \pkg{StabMatrix}, \pkg{StableEnvelope}, \pkg{TangentSpace}, \pkg{TautologicalBundle}, \pkg{ThomClass}, and \pkg{VirtualTangentSpace}.

\subsection{Input data}\label{input}
The data of an affine type $A_r$ quiver variety is contained in two dimension vectors: $\mathsf{v}$ and $\mathsf{w}$. In the Maple package, these are each described by a list of integers. For definiteness, we will illustrate the code here for the specific example of
\begin{align*}
&{\color{myred}\pkg{v:=[2,2,3]}:}\\
&{\color{myred}\pkg{w:=[2,1,0]}:}
\end{align*}

For the computation of stable envelopes, the other data that we need is a choice of chamber and a choice of polarization. The chamber is equivalent to an ordering of the equivariant parameters. As explained in this paper, we assume that the additional equivariant parameter $a$ related to $t_1$ and $t_2$ is infinitesimally small compared to the framing parameters. The framing parameter will be written as $\pkg{u[i,j]}$ for $1\leq i \leq r$ and $1 \leq j \leq \mathsf{v}_i$. 

A chamber will be stored as a $\pkg{chamb}$, which is a list such that
\begin{itemize}
    \item The first two items in the list are $\pkg{t[1]}$ and $\pkg{t[2]}$, in either order.
    \item Each framing parameter appears exactly once in the list.
\end{itemize}

For $\pkg{v}$ and $\pkg{w}$ above, one particular assignment of a chamber is given by
$$
{\color{myred}\pkg{chamb:=[t[2],t[1],u[1,1],u[1,2],u[2,1]]}}
$$

The list $\pkg{chamb}$ is interpreted in the following way:
\begin{align*}
\pkg{[t[i],t[j],\ldots]} &\iff t_i \text{ is repelling and } \,   t_j \text{ is attracting} \\
\pkg{[\ldots,u[i,j],\ldots,u[k,l]]} &\iff  \frac{u_{i,j}}{u_{k,l}} \text{ is repelling}
\end{align*}

Nakajima quiver varieties have a collection of natural polarizations. We will store this data as a list called $\pkg{arrows}$, which is a list of three lists with $1$'s or $-1$'s, where each of the three lists are of length $r$. The $i$th list specifies the choice of terms in the $i$th row of (\ref{tan}) in the following manner\footnote{We reindex the bundles in (\ref{tan}) so that the tautological bundles are $\mathcal{V}_1,\ldots,\mathcal{V}_r$}:

\begin{itemize}
    \item An entry of $1$ in the $j$th entry of the $1$st list indicates the choice of $t_2 Hom(\mathcal{V}_{j},\mathcal{V}_{j+1})$, while an entry of $-1$ indicates the choice of $t_1 Hom(\mathcal{V}_{j+1},\mathcal{V}_{j})$.

\item An entry of $1$ in the $j$th entry of the $2$nd list indicates the choice of $Hom(\mathcal{W}_{j},\mathcal{V}_{j})$, while an entry of $-1$ indicates $t_1 t_2 Hom(\mathcal{V}_{j},\mathcal{W}_{j})$.

\item An entry of $1$ in the $j$th entry of the $3$rd list indicates the choice of $Hom(\mathcal{V}_{j},\mathcal{V}_{j})$, while an entry of $-1$ indicates $t_1 t_2 Hom(\mathcal{V}_{j},\mathcal{V}_{j})$.

 \end{itemize}

We will refer to such a list of lists using the name $\pkg{arrows}$ because of the natural correspondence of the terms in the first two rows of (\ref{tan}) with arrows in the quiver.
 
For the running example in this section, the choice of
\begin{align*}
   {\color{myred}\pkg{arrows}\pkg{:=[[1,-1,1],[-1,1,1],[-1,-1,1]]}:}
\end{align*} 
corresponds to the polarization
\begin{align*}
     T^{1/2}&=t_2 Hom(\mathcal{V}_1,\mathcal{V}_2) + t_1 Hom(\mathcal{V}_3,\mathcal{V}_2) + t_2 Hom(\mathcal{V}_3,\mathcal{V}_1)  \\
    &+ t_1 t_2 Hom(\mathcal{V}_1,\mathcal{W}_1) + Hom(\mathcal{W}_2,\mathcal{V}_2) + Hom(\mathcal{W}_3,\mathcal{V}_3) \\
    &- t_1 t_2 Hom(\mathcal{V}_1,\mathcal{V}_1) - t_1t_2 Hom(\mathcal{V}_2,\mathcal{V}_2) - Hom(\mathcal{V}_3,\mathcal{V}_3)
\end{align*}

Now that we have described how the data of a quiver variety, a chamber, and a polarization are stored, we will now proceed to describe each of the procedures provided by the package.

\subsection{\pkg{FixedPoints}}\label{pkgFP}
The calling sequence is
$$
\pkg{FixedPoints(v,w,chamb)}
$$
where 
\begin{itemize}
    \item \pkg{v}, \pkg{w}, and \pkg{chamb} are as described in section \ref{input}. 
\end{itemize}
The procedure \pkg{FixedPoints(v,w,chamb)} returns the list of $\bT$-fixed points on the quiver variety determined by \pkg{v} and \pkg{w}, ordered by \pkg{chamb} as in section 6.3.2 in \cite{AOElliptic}. The matrix of the elliptic stable envelope (see \pkg{StabMatrix} below) is upper triangular with respect to this ordering.

As discussed previously, a $\bT$-fixed point is described by $(\mathsf{v},\mathsf{w})$-tuple of partitions. We store the data of a partition as a (possibly empty) list \pkg{L} of nonincreasing positive integers. If a partition corresponds to the $j$th framing at vertex $i$, then we store the data of this partition as $\pkg{[[L],i,j]}$.

For example, the variety determined by $\mathsf{v}=(2,2,3)$, $\mathsf{w}=(2,1,0)$ has 51 $\bT$-fixed points. The following gives an example:
\begin{align*}
    &{\color{myred}\pkg{v:=[2,2,3]:}} \\ &{\color{myred}\pkg{w:=[2,1,0]:}}\\ &{\color{myred}\pkg{chamb:=[t[2],t[1],u[1,1],u[1,2],u[2,1]]:}} \\
    &{\color{myred}\pkg{FixedPoints(v,w,chamb)[1];} }\\
    &{\color{myred}\pkg{FixedPoints(v,w,chamb)[35];}} \\
    & \quad {\color{blue} [[[\,],1,1],[[\,],1,2],[[3,1,1,1,1],2,1]]} \\
    & \quad {\color{blue} [[[3],1,1],[[2],1,2],[[1,1],2,1]]}
\end{align*}

% \subsection{\pkg{FixedPoints}}
% The procedure \pkg{FixedPoints} takes parameters $(\pkg{v},\pkg{w},\pkg{chamb})$ and returns the list of $\bT$-fixed points on the quiver variety, ordered by \pkg{chamb} as in section 6.3.2 in \cite{AOElliptic}. The matrix of the elliptic stable envelope (see \pkg{StabMatrix} below) is upper triangular with respect to this ordering.

% As discussed previously, a $\bT$-fixed point is described by $(\mathsf{v},\mathsf{w})$-tuple of partitions. In Maple, the data of a partition is stored as a (possibly empty) list \pkg{L} of nonincreasing positive integers. If a partition arises from the $j$th framing at vertex $i$, then we store the data of this based partition as $\pkg{[[L],i,j]}$.

% For example, the variety determined by $\mathsf{v}=(2,2,3)$, $\mathsf{w}=(2,1,0)$ has 51 $\bT$-fixed points. The following gives an example:
% \begin{align*}
%     &\pkg{v:=[2,2,3]:} \\ &\pkg{w:=[2,1,0]:}\\ &\pkg{chamb:=[t[2],t[1],u[1,1],u[1,2],u[2,1]]:} \\
%     &\pkg{FixedPoints(v,w,chamb)[1];} \\
%     &\pkg{FixedPoints(v,w,chamb)[35];} \\
%     & \quad {\color{blue} [[[\,],1,1],[[\,],1,2],[[3,1,1,1,1],2,1]]} \\
%     & \quad {\color{blue} [[[3],1,1],[[2],1,2],[[1,1],2,1]]}
% \end{align*}

\subsection{\pkg{VirtualTangentSpace}}
The calling sequence is
$$
\pkg{VirtualTangentSpace(v,w)}
$$
where
\begin{itemize}
    \item \pkg{v} and \pkg{w} are as described in section \ref{input}
\end{itemize}
The procedure \pkg{VirtualTangentSpace(v,w)} returns the virtual tangent space (\ref{tan}) on the variety determined by \pkg{v} and \pkg{w}, written in terms of the Chern roots of the tautological bundles: $\pkg{x[i,j]}$ where $1 \leq i \leq r$ and $1\leq j \leq \mathsf{v}_i$.

\subsection{\pkg{TautologicalBundle}}
The calling sequence is
$$
\pkg{TautologicalBundle(v,i)} 
$$
where
\begin{itemize}
    \item \pkg{v} is as in section \ref{mainthmintro}
    \item \pkg{i} is an integer between 1 and the length of \pkg{v}, inclusive.
\end{itemize}
The procedure \pkg{TautologicalBundle(v,i)} returns the $i$th tautological bundle of a variety with dimension vector \pkg{v} expressed in terms of its Chern roots. The description of the tautological bundles in terms of Chern roots does not depend on \pkg{w}. Hence we omit it as a parameter.

\subsection{\pkg{TangentSpace}}
The calling sequence is
$$
\pkg{TangentSpace(v,w,S)}
$$
where
\begin{itemize}
    \item \pkg{v} and \pkg{w} are as in section \ref{input}
    \item \pkg{S} is a list of based partitions indexing a fixed point, as in section \ref{pkgFP}
\end{itemize}
The procedure \pkg{TangentSpace(v,w,S)} returns the $\bT$-character of the tangent space of the variety determined by \pkg{v} and \pkg{w} at the fixed point indexed by $S$.

Consider the following example:
\begin{align*}
   &{\color{myred}\pkg{v:=[2,2,3]}:}\\ 
   &{\color{myred}\pkg{w:=[2,1,0]:}}\\ &{\color{myred}\pkg{chamb:=[t[2],t[1],u[1,1],u[1,2],u[2,1]]:}} \\
    &{\color{myred}\pkg{B:=FixedPoints(v,w,chamb):}} \\
    &{\color{myred}\pkg{TangentSpace(v,w,B[1]);}} \\
    & \quad {\color{blue} \frac{u_{2,1}}{u_{1,2}t_1} + \frac{u_{2,1}}{u_{1,2}t_2^2} + \frac{u_{2,1}}{u_{1,1} t_1} + \frac{u_{2,1}}{u_{1,1} t_2^2}+ \frac{t_1}{t_2^2} + \frac{t_1 t_2^{3} u_{1,1}}{u_{2,1}} + \frac{t_1^2 t_2 u_{1,2}}{u_{2,1}} + \frac{t_1 t_2^{3} u_{1,2}}{u_{2,1}} } \\
   & \quad {\color{blue} + t_2^{3} + \frac{t_1^{2} t_{2} u_{1,1}}{u_{2,1}} }
\end{align*}

\subsection{\pkg{Attracting} and \pkg{Repelling}}
The calling sequences are
\begin{align}
    &\pkg{Attracting(v,w,S,chamb)}\\
    &\pkg{Repelling(v,w,S,chamb)}
\end{align}
where
\begin{itemize}
    \item \pkg{v}, \pkg{w}, and \pkg{chamb} are as in section \ref{input}.
    \item \pkg{S} is a list of based partitions indexing a fixed point, as in section \ref{pkgFP}.
\end{itemize}
The procedure \pkg{Attracting(v,w,S,chamb)} return the attracting part with respect to \pkg{chamb} of the tangent space at \pkg{S} of the variety determined by \pkg{v} and \pkg{w}. Similarly, \pkg{Repelling(v,w,S,chamb)} returns the repelling part.

As an example, consider:
\begin{align*}
   &{\color{myred}\pkg{v:=[2,2,3]:} }\\ 
   &{\color{myred}\pkg{w:=[2,1,0]:}}\\ &{\color{myred}\pkg{chamb:=[t[2],t[1],u[1,1],u[1,2],u[2,1]]:} }\\
    &{\color{myred}\pkg{B:=FixedPoints(v,w,chamb):}} \\
    &{\color{myred}\pkg{Attracting(v,w,B[1],chamb);} }\\
    &{\color{myred}\pkg{Repelling(v,w,B[1],chamb);}}\\
    & \quad {\color{blue} \frac{u_{2,1}}{u_{1,2}t_1} + \frac{u_{2,1}}{u_{1,2}t_2^2} + \frac{u_{2,1}}{u_{1,1} t_1} + \frac{u_{2,1}}{u_{1,1} t_2^2}+ \frac{t_1}{t_2^2}\ } \\
  & \quad {\color{blue}  \frac{t_1 t_2^{3} u_{1,1}}{u_{2,1}} + \frac{t_1^2 t_2 u_{1,2}}{u_{2,1}} + \frac{t_1 t_2^{3} u_{1,2}}{u_{2,1}} + t_2^{3} + \frac{t_1^{2} t_{2} u_{1,1}}{u_{2,1}}}
\end{align*}

\subsection{\pkg{ChamberExample}}
The calling sequence is
$$
\pkg{ChamberExample(v,w)}
$$
where
\begin{itemize}
    \item \pkg{v} and \pkg{w} are as described in section \ref{input}.
\end{itemize}
The procedure $\pkg{ChamberExample(v,w)}$ returns an example of a chamber on the variety determined by \pkg{v} and \pkg{w}. The first element in the returned list is $\pkg{t[2]}$. The equivariant parameters are ordered first based on the vertex from which they originate, and second by the number of the framing they correspond to.
\begin{align*}
   &{\color{myred}\pkg{v:=[2,2,3]:}} \\ 
   &{\color{myred}\pkg{w:=[2,1,0]:}}\\ 
   & {\color{myred}\pkg{ChamberExample(v,w);}}\\
    & \quad {\color{blue} [t_2,t_1,u_{1,1},u_{1,2},u_{2,1}]\ }
\end{align*}

This procedure is mainly intended for convenience, so that one does not have to manually input a chamber once the dimensions are chosen.

\subsection{\pkg{PolExample}}
The calling sequence is
$$
\pkg{PolExample(v,w)}
$$
where
\begin{itemize}
    \item \pkg{v} and \pkg{w} are as in section \ref{input}.
\end{itemize}
The procedure \pkg{PolExample(v,w)} returns a choice of polarization on the variety determined by \pkg{v} and \pkg{w} in the form explained for $\pkg{arrows}$ in section \ref{input}.

\begin{align*}
   &{\color{myred}\pkg{v:=[2,2,3]}:} \\
   &{\color{myred}\pkg{w:=[2,1,0]}:}\\ 
   &{\color{myred} \pkg{PolExample(v,w)};}\\
    & \quad {\color{blue} [[1,1,1],[1,1,1],[1,1,1]]\ }
\end{align*}

As with $\pkg{ChamberExample}$, this procedure is mainly intended for convenience.

\subsection{\pkg{NormalForm}}
The calling sequence is
$$
\pkg{NormalForm(f)}
$$
where
\begin{itemize}
    \item $f$ is some algebraic expression, which may include terms like \pkg{theta(x)} and \pkg{phi(x)} where \pkg{x} is some monomial in the equivariant parameters, the Chern roots of the tautological bundles, and the variable $q$.
\end{itemize}
The procedure \pkg{NormalForm(f)} uses (\ref{thetaphi2}) to return an equal expression with the minimal powers of $q$ inside the arguments of $\pkg{theta}$ or $\pkg{phi}$ terms. While \pkg{NormalForm} does not use more complicated identities like the 3-term theta function identity, it can nevertheless be useful for algebraically checking if two expressions involving such functions are equal. For example:

\begin{align*}
   &{\color{myred}\pkg{f:=theta(q*a)+theta(b/q):}}\\
   &{\color{myred} \pkg{g:=theta(1):}} \\
   &{\color{myred}\pkg{h:=theta(q)/theta(q} \caret \pkg{2):}} \\
   &{\color{myred} \pkg{NormalForm(f);} }\\
   &{\color{myred} \pkg{NormalForm(g);}} \\
   &{\color{myred}\pkg{NormalForm(h);}} \\
   & \quad {\color{blue} -\frac{b \theta(b) a + \theta(a)}{\sqrt{q} a}} \\
   & \quad {\color{blue} 0} \\
   & \quad {\color{blue} -q^{3/2}}
\end{align*}

\subsection{\pkg{ThomClass}}
The calling sequence is
$$
\pkg{ThomClass(f)}
$$
where
\begin{itemize}
    \item \pkg{f} is a Laurent polynomial in the Chern roots of the tautological bundles and equivariant parameters.
\end{itemize}
The procedure \pkg{ThomClass(f)} returns \pkg{theta(f)} if \pkg{f} is a Laurent monomial, and is defined by multiplicativity in general. 

\begin{align*}
    & {\color{myred}\pkg{ThomClass(a+b-c);}} \\
     &  \quad {\color{blue} \frac{\theta(a) \theta(b)}{\theta(c)}}
\end{align*}

\subsection{\pkg{Polarization}}
The calling sequence is
$$
\pkg{Polarization(v,w,arrows)}
$$
where
\begin{itemize}
    \item \pkg{v}, \pkg{w}, and \pkg{arrows} are as in section \ref{input}
\end{itemize}
The procedure \pkg{Polarization(v,w,arrows)} returns the polarization determined by \pkg{arrows} of the variety determined by \pkg{v} and \pkg{w}, expressed in terms of the equivariant parameters and the Chern roots of the tautological bundles.

\subsection{\pkg{Quasiperiods}}
The calling sequence is
$$
\pkg{Quasiperiods(v,w,S,arrows,chamb)}
$$
where
\begin{itemize}
    \item \pkg{v}, \pkg{w}, \pkg{arrows}, and \pkg{chamb} are as in section \ref{input}.
    \item \pkg{S} is the data describing a fixed point as in section \ref{pkgFP}.
\end{itemize}
The procedure \pkg{Quasiperiods(v,w,S,arrows,chamb)} returns a rational expression in theta functions whose quasiperiods match that of the elliptic stable envelope of the fixed point \pkg{S} with respect to the polarization and chamber specified by $\pkg{arrows}$ and $\pkg{chamb}$.

\begin{align*}
      &{\color{myred}\pkg{v:=[1,1]:} }\\ 
      &{\color{myred}\pkg{w:=[1,1]:}}\\ 
   & {\color{myred}\pkg{arrows:=[[1,1],[1,1],[1,1]]:}}\\
   & {\color{myred}\pkg{chamb:=[t[2],t[1],u[1,1],u[2,1]:}} \\
   & {\color{myred}\pkg{S:=FixedPoints(v,w,chamb)[1];}} \\
   & {\color{myred}\pkg{qp:=Quasiperiods(v,w,S,arrows,chamb);}} \\
   &{\color{myred} \pkg{NormalForm(subs(x[2,1]=q*x[2,1],qp)/qp);}} \\
   & \quad {\color{blue} S:=[[[\,],1,1],[[1,1],2,1]]} \\
    & \quad {\color{blue} qp:=\frac{\theta\left(\frac{t_2 x_{2,1}}{x_{2,1}}\right) \theta\left(\frac{t_2 x_{1,1}}{x_{2,1}}\right) \theta\left(\frac{x_{1,1}}{u_{1,1}}\right) \theta\left(\frac{x_{2,1}}{u_{2,1}}\right) \theta(x_{1,1} z_1) \theta(x_{2,1} z_2) \theta\left(\frac{u_{2,1}}{t_2}\right) \theta(u_{2,1}) \theta(t_{1} t_2)^{2} \theta\left(\frac{u_{2,1}}{u_{1,1} t_2}\right)}{\theta(1)^2 \theta(x_{1,1}) \theta(x_{2,1}) \theta\left(\frac{u_{2,1} x_1}{t_2}\right) \theta(u_{2,1} z_2) \theta\left( \frac{t_1 u_{2,1}}{u_{1,1}} \right)}\ } \\
    & \quad {\color{blue} -\frac{x_{1,1}^2 u_{2,1}}{q^{3/2} x_{2,1}^{3} z_2}
    }
\end{align*}
This expression is not always mathematically well-defined, as can be seen from the ${\color{blue} \theta(1)}$ in the denominator. Nevertheless, the quasiperiods of such an expression are perfectly well-defined.

\subsection{\pkg{Restrict}}
The calling sequence is
$$
\pkg{Restrict(v,w,S)}
$$
where
\begin{itemize}
    \item \pkg{v} and \pkg{w} are as in section \ref{input}.
    \item \pkg{S} indexes a fixed point, as in section \ref{pkgFP}.
\end{itemize}
The procedure \pkg{Restrict(v,w,S)} returns the $\bT$-weights that the Chern roots of the tautological bundles on the variety determined by \pkg{v} and \pkg{w} restrict to at the fixed point determined by \pkg{S}.

\begin{align*}
     &{\color{myred}\pkg{v:=[1,1]:} }\\ 
     &{\color{myred}\pkg{w:=[1,1]:}}\\ 
   & {\color{myred}\pkg{arrows:=[[1,1],[1,1],[1,1]]:}}\\
   & {\color{myred}\pkg{chamb:=[t[2],t[1],u[1,1],u[2,1]:}} \\
   & {\color{myred}\pkg{S:=FixedPoints(v,w,chamb)[1];} }\\
   & {\color{myred}\pkg{Res:=Restrict(v,w,S);}} \\
   & {\color{myred}\pkg{subs(Res,VirtualTangentSpace(v,w));}}\\
   & {\color{myred}\pkg{TangentSpace(v,w,S);} }\\
   & \quad {\color{blue} S:=[[[\,],1,1],[[1,1],2,1]]} \\
   & \quad {\color{blue} Res:=\left\{x_{1,1}=\frac{u_{2,1}}{t_2}, x_{2,1}=u_{2,1}\right\}} \\
   & \quad {\color{blue} t_2^2+\frac{u_{2,1}}{u_{1,1} t_2} + \frac{t_1}{t_2} + \frac{t_1 t_2^2 u_{1,1}}{u_{2,1}} } \\
    & \quad {\color{blue} t_2^2+\frac{u_{2,1}}{u_{1,1} t_2} + \frac{t_1}{t_2} + \frac{t_1 t_2^2 u_{1,1}}{u_{2,1}} }
\end{align*}

\subsection{\pkg{StableEnvelope}}
The calling sequence is
$$
\pkg{StableEnvelope(v,w,S,arrows,chamb)}
$$
where
\begin{itemize}
    \item \pkg{v}, \pkg{w}, \pkg{arrows}, and \pkg{chamb} are as in section \ref{input}.
    \item \pkg{S} is the data describing a fixed point as in section \ref{pkgFP}.
\end{itemize}
The procedure \pkg{StableEnvelope(v,w,S,arrows,chamb)} returns the elliptic stable envelope of the fixed point $\pkg{S}$ of the variety determined by \pkg{v} and \pkg{w} with the corresponding chamber and polarization given by \pkg{chamb} and \pkg{arrows}, as written in Theorem \ref{mainthm}. This is one of the main procedures.

\begin{align*}
    & {\color{myred}\pkg{v:=[2]:} }\\
    &  {\color{myred}\pkg{w:=[1]:}} \\
    &  {\color{myred}\pkg{chamb:=[t[2],t[1],u[1,1]]:}} \\
    &  {\color{myred}\pkg{arrows:=[[1],[1],[1]]:}} \\
    &  {\color{myred}\pkg{S:=FixedPoints(v,w,chamb)[1];}}\\
    &  {\color{myred}\pkg{StableEnvelope(v,w,S,chamb,arrows);} } \\
    & \quad {\color{blue} S:=[[[1,1],1,1]]} \\
        & \quad {\color{blue} \frac{\theta\left(\frac{t_1 t_2 u_{1,1}}{x_{1,2}}\right)\theta(t_2)^2\theta\left(\frac{t_2 x_{1,1}}{x_{1,2}}\right)\theta\left(\frac{x_{1,2} t_2 z_1}{x_{1,1}}\right)\theta\left(\frac{x_{1,1}z_1^2 t_1 t_2}{u_{1,1}}\right)}{\theta\left(\frac{x_{1,1}}{x_{1,2}}\right)\theta\left(\frac{t_1 t_2 x_{1,1}}{x_{1,2}}\right)\theta(z_1) \theta(z_1^2 t_1 t_2)}} \\
       & \quad \quad {\color{blue} +\frac{\theta\left(\frac{t_1 t_2 u_{1,1}}{x_{1,1}}\right)\theta(t_2)^2\theta\left(\frac{t_2 x_{1,2}}{x_{1,1}}\right)\theta\left(\frac{x_{1,1} t_2 z_1}{x_{1,2}}\right)\theta\left(\frac{x_{1,2}z_1^2 t_1 t_2}{u_{1,1}}\right)}{\theta\left(\frac{x_{1,2}}{x_{1,1}}\right)\theta\left(\frac{t_1 t_2 x_{1,2}}{x_{1,1}}\right)\theta(z_1) \theta(z_1^2 t_1 t_2)}}
\end{align*}

\subsection{\pkg{StabMatrix}}
The calling sequence is
$$
\pkg{StabMatrix(v,w,arrows,chamb)}
$$
where
\begin{itemize}
    \item \pkg{v}, \pkg{w}, \pkg{arrows}, and \pkg{chamb} are as in section \ref{input}.
\end{itemize}
The procedure \pkg{StabMatrix(v,w,arrows,chamb)} returns the matrix of restrictions of the elliptic stable envelope of the variety determined by \pkg{v} and \pkg{w} with respect to the chamber \pkg{chamb} and polarization \pkg{arrows}.

\begin{align*}
    &  {\color{myred}\pkg{v:=[2]:}} \\
    &  {\color{myred}\pkg{w:=[1]:} }\\
    &  {\color{myred}\pkg{chamb:=[t[2],t[1],u[1,1]]:}} \\
    &  {\color{myred}\pkg{arrows:=[[1],[1],[1]]:}} \\
    &  {\color{myred}\pkg{StabMatrix(v,w,arrows,chamb);}} \\
    & \quad {\color{blue} \begin{bmatrix}
    \theta(t_2) \theta(t_2^2) & \frac{\theta(t_1t_2)\theta(t_2) \theta\left(\frac{t_1}{t_2}\right) \theta(t_1t_2 z_1) \theta(t_2 z_1^2)}{\theta(t_1)\theta(z_1)\theta(z_1^2 t_1 t_2)}-\frac{\theta(t_2)^2\theta(t_1t_2)\theta\left(\frac{t_1}{t_2z_1}\right)}{\theta(t_1)\theta(z_1)} \\
    0 & -\theta(t_2)\theta\left(\frac{t_1}{t_2}\right)
    \end{bmatrix}}
\end{align*}

\subsection{\pkg{KStableEnvelope}}
The calling sequence is 
$$
\pkg{KStableEnvelope(v,w,S,arrows,chamb,s)}
$$
where
\begin{itemize}
    \item \pkg{v}, \pkg{w}, \pkg{arrows}, \pkg{chamb} are as in section \ref{input}.
    \item \pkg{S} is the data describing a fixed point as in section \ref{pkgFP}.
    \item \pkg{s} is a list of real numbers of length equal to the length of \pkg{v}.
\end{itemize}
The procedure \pkg{KStableEnvelope(v,w,S,arrows,chamb,s)} returns the $K$-theoretic stable envelope of the fixed point \pkg{S} of the variety determined by \pkg{v} and \pkg{w} with the corresponding chamber, polarization, and slope given by \pkg{chamb}, \pkg{arrows}, and \pkg{s}. For a description of slope, see section \ref{Ktheory}.

\begin{align*}
    & {\color{myred}\pkg{v:=[2]:} }\\
    &  {\color{myred}\pkg{w:=[1]:}} \\
    &  {\color{myred}\pkg{chamb:=[t[2],t[1],u[1,1]]:}} \\
    &  {\color{myred}\pkg{arrows:=[[1],[1],[1]]:}} \\
     &  {\color{myred}\pkg{s:=[1/3]:}} \\
    &  {\color{myred}\pkg{S:=FixedPoints(v,w,chamb)[1];}}\\
    &  {\color{myred}\pkg{KStableEnvelope(v,w,S,chamb,arrows,s);} } \\
    & \quad {\color{blue} S:=[[[1,1],1,1]]} \\
        & \quad {\color{blue} \frac{\frac{(t_2x_{1,1}-x_{1,2})(t_2-1)^{2}(t_1t_2u_{1,1}-x_{1,2})\sqrt{x_{1,1}}}{\sqrt{x_{1,2}}(t_1t_2x_{1,1}-x_{1,2})(x_{1,1}-x_{1,2})t_2^2}+\frac{(t_2x_{1,2}-x_{1,1})(t_2-1)^{2}(t_1t_2u_{1,1}-x_{1,1})\sqrt{x_{1,2}}}{\sqrt{x_{1,1}}(t_1t_2x_{1,2}-x_{1,1}(x_{1,2}-x_{1,1})t_2^2}}{\sqrt{t_2^3}} }
\end{align*}

\subsection{\pkg{KStabMatrix}}
The calling sequence is
$$
\pkg{KStabMatrix(v,w,arrows,chamb,s)}
$$
where
\begin{itemize}
    \item \pkg{v}, \pkg{w}, \pkg{arrows}, and \pkg{chamb} are as in section \ref{input}.
    \item \pkg{s} is a list of real numbers of length equal to the length of \pkg{v}.
\end{itemize}
The procedure \pkg{KStabMatrix(v,w,arrows,chamb,s)} returns the matrix of restrictions of the $K$-theoretic stable envelope of the variety determined by \pkg{v} and \pkg{w} with respect to the chamber \pkg{chamb}, polarization \pkg{arrows}, and slope \pkg{s}.

\begin{align*}
    &  {\color{myred}\pkg{v:=[2]:}} \\
    &  {\color{myred}\pkg{w:=[1]:} }\\
    &  {\color{myred}\pkg{chamb:=[t[2],t[1],u[1,1]]:}} \\
    &  {\color{myred}\pkg{arrows:=[[1],[1],[1]]:}} \\
     & {\color{myred}\pkg{s1:=[1/3]}:}\\
      & {\color{myred}\pkg{s2:=[1/2]}:}\\
    &  {\color{myred}\pkg{KStabMatrix(v,w,arrows,chamb,s1);}} \\
     &  {\color{myred}\pkg{KStabMatrix(v,w,arrows,chamb,s2);}} \\
    & \quad {\color{blue} \begin{bmatrix}
  \frac{ (t_2+1)(t_2-1)^2}{t_2^3} & \frac{(t_1t_2-1)(t_2-1)}{t_2^{5/2} \sqrt{t_1}}\\
  0 & -\frac{(t_1-t_2)(t_2-1)}{t_2^3}
    \end{bmatrix}} \\
    & \quad {\color{blue} \begin{bmatrix}
  \frac{ (t_2+1)(t_2-1)^2}{t_2^3} & \frac{(t_1t_2-1)(t_2-1) \sqrt{t_1}(t_2^2 z_1^2-1)}{t_2^{7/2} (t_1 t_2 z_1^2-1)}\\
  0 & -\frac{(t_1-t_2)(t_2-1)}{t_2^3}
    \end{bmatrix}}
\end{align*}

\printbibliography

@article{MO,
author = {Maulik, Davesh and Okounkov, Andrei},
year = {2012},
month = {11},
pages = {},
title = {Quantum Groups and Quantum Cohomology},
volume = {408},
journal = {Astérisque}
}

@incollection{pcmilect,
      author         = "Okounkov, Andrei",
      title          = "{Lectures on K-theoretic computations in enumerative
                        geometry}",
    booktitle= {Geometry of Moduli Spaces and Representation Theory},
    publisher={American Mathematical Society},
    series={IAS/Park City Mathematics Series},
      year           = "2017",
      volume={24}
}

@preamble{
   "\def\cprime{$'$} "
}

@article {Nak1,
    AUTHOR = {Nakajima, Hiraku},
     TITLE = {Quiver varieties and {K}ac-{M}oody algebras},
   JOURNAL = {Duke Math. J.},
    VOLUME = {91},
      YEAR = {1998},
    NUMBER = {3},
     PAGES = {515-560},
}

@article {NakALE,
    AUTHOR = {Nakajima, Hiraku},
     TITLE = {Instantons on {ALE} spaces, quiver varieties, and
              {K}ac-{M}oody algebras},
   JOURNAL = {Duke Math. J.},
  FJOURNAL = {Duke Mathematical Journal},
    VOLUME = {76},
      YEAR = {1994},
    NUMBER = {2},
     PAGES = {365--416},
    %   ISSN = {0012-7094},
     CODEN = {DUMJAO},
   MRCLASS = {53C25 (17B67 58D27 58E15)},
  MRNUMBER = {1302318 (95i:53051)},
MRREVIEWER = {Andrew Dancer},
       DOI = {10.1215/S0012-7094-94-07613-8},
      
}

@online{OS,
      author         = "Okounkov, Andrei and Smirnov, Andrey",
      title          = "{Quantum difference equation for Nakajima varieties}",
      year           = "2016",
      eprint         = "1602.09007",
      archivePrefix  = "arXiv",
      primaryClass   = "math-ph",
    
}

@phdthesis{Shenfeld,
	AUTHOR = {Shenfeld, Daniel},
	TITLE = {Abelianization of stable envelopes in symplectic resolutions},
	NOTE = {ProQuest LLC, Ann Arbor, MI},
	school={Princeton University},
	YEAR = {2013},
	ISBN = {978-1303-45685-5}
}

@online{AOElliptic,
       author = {{Aganagic}, Mina and {Okounkov}, Andrei},
        title = "{Elliptic stable envelopes}",
      journal = {arXiv e-prints},
         year = "2016",
        month = "4",
          eid = {arXiv:1604.00423},
archivePrefix = {arXiv},
       eprint = {1604.00423v4},
 primaryClass = {math.AG},
       adsurl = {https://ui.adsabs.harvard.edu/abs/2016arXiv160400423A}
}

@incollection {GinzburgLectures,
    AUTHOR = {Ginzburg, Victor},
     TITLE = {Lectures on {N}akajima's quiver varieties},
 BOOKTITLE = {Geometric methods in representation theory. {I}},
    SERIES = {S\'emin. Congr.},
    VOLUME = {24},
     PAGES = {145--219},
 PUBLISHER = {Soc. Math. France, Paris},
      YEAR = {2012},
   MRCLASS = {14L24 (16G20 17B67)},
  MRNUMBER = {3202703},
MRREVIEWER = {Xueqing Chen},
}

@article{SmirnovElliptic,
author = {{Smirnov}, Andrey},
year = {2019},
month = {12},
pages = {},
title = {Elliptic stable envelope for Hilbert scheme of points in the plane},
volume = {26},
journal = {Selecta Math.}
}

@incollection {ProudHyper,
	AUTHOR = {Proudfoot, Nicholas J.},
	TITLE = {A survey of hypertoric geometry and topology},
	BOOKTITLE = {Toric topology},
	SERIES = {Contemp. Math.},
	VOLUME = {460},
	PAGES = {323--338},
	PUBLISHER = {Amer. Math. Soc., Providence, RI},
	YEAR = {2008}
}

@article{RTV,
author = {Rimányi, R. and Tarasov, V. and Varchenko, A.},
year = {2017},
month = {05},
pages = {},
title = {Elliptic and K-theoretic stable envelopes and Newton polytopes},
volume = {25},
journal = {Selecta Mathematica}
}

@article {kirv,
	AUTHOR = {McGerty, Kevin and Nevins, Thomas},
	TITLE = {Kirwan surjectivity for quiver varieties},
	JOURNAL = {Invent. Math.},
	VOLUME = {212},
	YEAR = {2018},
	NUMBER = {1},
	PAGES = {161--187}
}

@online{MirSym1,
	author = {{Rim{\'a}nyi}, Rich{\'a}rd and {Smirnov}, Andrey and {Varchenko}, Alexand and {Zhou}, Zijun},
	title = "{3d Mirror Symmetry and Elliptic Stable Envelopes}",
	journal = {arXiv e-prints},
	year = "2019",
	month = 2,
	eid = {arXiv:1902.03677},
	pages = {arXiv:1902.03677},
	archivePrefix = {arXiv},
	eprint = {1902.03677},
	primaryClass = {math.AG}
}

@article{MirSym2,
	author = {{Rim{\'a}nyi}, R. and {Smirnov}, A. and {Varchenko}, A. and {Zhou}, Z.},
	title = "{Three dimensional mirror self-symmetry of the cotangent bundle of the full flag variety}",
	journal = {SIGMA},
	volume = {15},
	year = "2019",
	month = 11,
	pages = {1-22}
}

@article{NakBow,
  title={Cherkis bow varieties and Coulomb branches of quiver gauge theories of affine type A},
  author={H. Nakajima and Yuuya Takayama},
  journal={Selecta Mathematica},
  year={2016},
  volume={23},
  pages={2553-2633}
}

@article{Cherk1,
   title={Moduli Spaces of Instantons on the Taub-NUT Space},
   volume={290},
   DOI={10.1007/s00220-009-0863-8},
   number={2},
   journal={Communications in Mathematical Physics},
   publisher={Springer Science and Business Media LLC},
   author={Cherkis, Sergey A.},
   year={2009},
   month={6},
   pages={719–736}
}

@article{Cherk2,
      title={Instantons on the Taub-NUT Space}, 
      author={Sergey A. Cherkis},
      year={2010},
      journal={Adv. Theor. Math. Phys.},
      number={14},
      pages={609-642},
      month={4}
}

@article{Cherk3,
	author = {{Cherkis}, Sergey A.},
	title = "{Instantons on Gravitons}",
	journal = {Communications in Mathematical Physics},
	year = "2011",
	month = "9",
	volume = {306},
	number = {2},
	pages = {449-483},
	doi = {10.1007/s00220-011-1293-y},
	archivePrefix = {arXiv},
	eprint = {1007.0044},
	primaryClass = {hep-th},
	adsurl = {https://ui.adsabs.harvard.edu/abs/2011CMaPh.306..449C}
}

@article{dinksmir2,
    author = {Dinkins, Hunter and Smirnov, Andrey},
    title = "{Quasimaps to Zero-Dimensional $A_{\infty}$-Quiver Varieties}",
    journal = {International Mathematics Research Notices},
    year = {2020},
    month = {06},
    doi = {10.1093/imrn/rnaa129}
}

@online{dinksmir3,
       author = {{Dinkins}, Hunter and {Smirnov}, Andrey},
        title = "{Capped vertex with descendants for zero dimensional $A_{\infty}$ quiver varieties}",
      journal = {arXiv e-prints},
         year = 2020,
        month = may,
          eid = {arXiv:2005.12980},
        pages = {arXiv:2005.12980},
archivePrefix = {arXiv},
       eprint = {2005.12980},
 primaryClass = {math.AG},
       adsurl = {https://ui.adsabs.harvard.edu/abs/2020arXiv200512980D}
}

@online{KS2,
       author = {{Kononov}, Yakov and {Smirnov}, Andrey},
        title = "{Pursuing quantum difference equations II: 3D-mirror symmetry}",
      journal = {arXiv e-prints},
         year = 2020,
        month = aug,
          eid = {arXiv:2008.06309},
        pages = {arXiv:2008.06309},
archivePrefix = {arXiv},
       eprint = {2008.06309},
 primaryClass = {math.AG},
       adsurl = {https://ui.adsabs.harvard.edu/abs/2020arXiv200806309K}
}

@online{mstoric,
       author = {{Smirnov}, Andrey and {Zhou}, Zijun},
        title = "{3d Mirror Symmetry and Quantum $K$-theory of Hypertoric Varieties}",
      journal = {arXiv e-prints},
         year = 2020,
        month = may,
          eid = {arXiv:2006.00118},
        pages = {arXiv:2006.00118},
archivePrefix = {arXiv},
       eprint = {2006.00118},
 primaryClass = {math.AG},
       adsurl = {https://ui.adsabs.harvard.edu/abs/2020arXiv200600118S}
}

@article{dinkms1,
       author = {{Dinkins}, Hunter},
        title = "{Symplectic Duality of $T^*Gr(k,n)$}",
      journal = {Mathematical Research Letters},
         year = 2021,
        pages = {to appear}
}

@online{msflag,
       author = {{Dinkins}, Hunter},
        title = "{3d mirror symmetry of the cotangent bundle of the full flag variety}",
      journal = {arXiv e-prints},
         year = 2020,
        month = nov,
          eid = {arXiv:2011.08603},
        pages = {arXiv:2011.08603},
archivePrefix = {arXiv},
       eprint = {2011.08603},
 primaryClass = {math.AG},
       adsurl = {https://ui.adsabs.harvard.edu/abs/2020arXiv201108603D}
}

@online{indstab1,
       author = {{Okounkov}, Andrei},
        title = "{Inductive construction of stable envelopes}",
      journal = {arXiv e-prints},
         year = 2020,
        month = jul,
          eid = {arXiv:2007.09094},
        pages = {arXiv:2007.09094},
archivePrefix = {arXiv},
       eprint = {2007.09094},
 primaryClass = {math.AG},
       adsurl = {https://ui.adsabs.harvard.edu/abs/2020arXiv200709094O}
}

@online{RSbows,
       author = {{Rimanyi}, R. and {Shou}, Y.},
        title = "{Bow varieties---geometry, combinatorics, characteristic classes}",
      journal = {arXiv e-prints},
         year = 2020,
        month = dec,
          eid = {arXiv:2012.07814},
        pages = {arXiv:2012.07814},
archivePrefix = {arXiv},
       eprint = {2012.07814},
 primaryClass = {math.AG}
}

@misc{indstab2,
      title={Nonabelian stable envelopes, vertex functions with descendents, and integral solutions of $q$-difference equations}, 
      author={Andrei Okounkov},
      year={2021},
      eprint={2010.13217},
      archivePrefix={arXiv},
      primaryClass={math.AG}
}

@unpublished{dinksmir4,
title={Euler characteristic of stable envelopes},
 author={{Dinkins}, Hunter and {Smirnov}, Andrey},
 note={In preparation}
}

@misc{Botta,
      title={Shuffle products for elliptic stable envelopes of Nakajima varieties}, 
      author={Tommaso Maria Botta},
      year={2021},
      eprint={2104.00976},
      archivePrefix={arXiv},
      primaryClass={math.AG}
}

@article{RWElliptic,
author = {Rimányi, Richárd and Weber, Andrzej},
year = {2021},
month = {02},
title = {Elliptic classes on Langlands dual flag varieties},
journal = {Communications in Contemporary Mathematics},
doi = {10.1142/S0219199721500140}
}

\newpage

\noindent
Hunter Dinkins\\
Department of Mathematics,\\
University of North Carolina at Chapel Hill,\\
Chapel Hill, NC 27599-3250, USA\\
hdinkins@live.unc.edu 

\end{document}